\documentclass[11pt, twoside]{article}
\usepackage{amsfonts}
\usepackage{amsmath}
\usepackage{amsthm}
\usepackage[mathscr]{euscript}
\usepackage{mathrsfs}
\usepackage{graphicx}
\usepackage{subfigure}
\usepackage{varwidth}
\usepackage{color}\usepackage{enumitem}
\usepackage{url}
\usepackage{authblk}

\newtheorem{prop}{\textbf{Proposition}}[section]
\newtheorem{definition}{\textbf{Definition}}[section]
\newtheorem{theo}{\textbf{Theorem}}[section]
\newtheorem{lem}{\textbf{Lemma}}[section]

\newtheorem{remark}{\textbf{Remark}}[section]

\newcommand\ee{\end{equation}}
\newcommand\bes{\begin{eqnarray}}
\newcommand\ees{\end{eqnarray}}
\newcommand\bess{\begin{eqnarray*}}
\newcommand\eess{\end{eqnarray*}}

\newcommand\nm{{\nonumber}}

\numberwithin{equation}{section}

\begin{document}
\setlength{\baselineskip}{16pt}

\title{\bf $L^{p}$-$L^{q}$ estimates of the heat kernels on graphs with applications to a parabolic system \footnote{The author was supported by NSFC Grant 12201184.}}

\author[b]{Yuanyang Hu\footnote{Corresponding author. {\sl E-mail}: yuanyhu@henu.edu.cn}}

\affil[b]{\small School of Mathematics and Statistics, Henan University, Kaifeng 475004, China}\vspace{2mm}

%\date{\today}
\maketitle

\begin{quote}
	\noindent{\bf Abstract.} 
Let $G=(V, E)$ be a locally finite connected graph satisfying curvature-dimension conditions ($CDE(n, 0)$ or its strengthened version $CDE'(n, 0))$) and polynomial volume growth conditions of degree $m$. We systematically establish sharp $L^{p}$-bounds and decay-type $L^{p}$-$L^{q}$ estimates for heat operators on $G$, accommodating both bounded and unbounded Laplacians. The analysis utilizes Li-Yau-type Harnack inequalities and geometric completeness arguments to handle degenerate cases. As a key application, we prove the existence of global solutions to a semilinear parabolic system on $G$ under critical exponents governed by volume growth dimension $m$. 

	
\end{quote}

\section{Introduction}
Over the years, $L^{p}$ boundness and $L^{p}$-$L^{q}$ boundness of the heat semigroup $(e^{-t\Delta})_{t>0}$ have  attracted widespread attention on miscellaneous space. 
Taylor \cite{Taylor} studied $L^{p}$ estimates of certain functions of the Laplace operator on a complete Riemannian manifold with $C^{\infty}$ bounded geometry,
while Cowling-Giulini-Meda \cite{Cowling-Giulini-Meda} established the $L^{p}$-$L^{q}$ estimates for functions of the Laplace-Beltrami operator on a connected noncompact semisimple Lie group. 
In the discrete setting,
Cowling-Meda-Setti \cite{Cowling-Meda-Setti} used techniques from harmonic analysis to find the $L^{p}$-$L^{q}$ estimates of the heat operator on a homogeneous tree $\mathfrak{X}$ of degree $N$ and discovered very close analogies with diffusions on hyperbolic spaces. 
Abundant authors have shown that there is a close relationship between the discrete heat kernel of the graph Laplace operator on the discretization and that of the Laplace-Beltrami operator on the manifold (see, e.g., M. Kanai \cite{KM} and N. Th. Varopoulos \cite{V}). 
Recently, the heat kernels on graphs have attracted a lot of attention as an analogues of diffusions on manifolds (see, e.g., I. Chavel-E. A. Feldman \cite{CF} and M. M. H. Pang \cite{MM}). 

In this paper, we investigate heat kernel estimates on locally finite connected graphs $G=(V,E)$ satisfying curvature dimension condition and polynomial volume growth condition. 
We systematically analyze heat kernel estimates on locally finite graphs under two curvature-dimension frameworks: the standard $CDE(n,0)$ and its strengthened variant $CDE'(n,0)$, where $CDE'(n,0)$ implies $CDE(n,0)$. For graphs satisfying $CDE(n,0)$ and $D_{\mu} < \infty$ (bounded Laplacians), utilizing the Li-Yau Harnack-type estimate established by Bauer-Horn-Lin-Lippner-Mangoubi-Yau in \cite{BHLLMY}, we establish $L^p$-$L^q$ bounds with decay rates $t^{-m/2(1/p-1/q)}$, governed by polynomial volume growth $\mathcal{V}(x,R) \geq C R^m$ (Theorems 2.1-2.2). 
Conversely, for complete graphs satisfying the strengthened $CDE'(n,0)$ condition, we turn to the Li-Yau-type inequality of Hua-Lin \cite{HL}, which dispenses with the $D_{\mu} < \infty$ requirement by leveraging geometric completeness. This methodological shift allows us to establish the same $L^p$-$L^q$ estimates of the heat semigroup for unbounded Laplacians (Theorem 2.3).

In the pioneering work of Fujita \cite{Fujita}, he studied the Cauchy problem
\begin{equation}\label{1.1}
	\begin{cases}
		u_t=\Delta_{\mathbb{R}^N}u+u^p~ &\text{in}~ \mathbb{R}^N \times(0,\infty), \\ u(0,x)=u_0(x)~ &\text{in}\;\;\mathbb{R}^N,
\end{cases}\end{equation}
where $\Delta_{\mathbb{R}^N}$ is the Laplace operator, $p>1$ and $u_0(x)$ is a nonnegative nontrivial bounded and continuous function on $\mathbb{R}^N$.
Define $p_{F}:=1+\frac{2}{N}$, which is called Fujita exponent. Fujita established the following results:
\begin{itemize}
\item If $1<p<p_{F}$, then any nonnegative solutions of \eqref{1.1} blows up in finite time;
\item If $p>p_{F}$, then \eqref{1.1} admits a global nonnegative solution for sufficiently small $u_0$.
\end{itemize}
Escobedo and Herrero \cite{EH} studied the weakly coupled system:
\begin{equation}\label{1.2}
\begin{cases}u_t=\Delta_{\mathbb{R}^N} u+v^p, & x\in \mathbb{R}^N,~ t>0, \\ v_t=\Delta_{\mathbb{R}^N} v+u^q, & x\in \mathbb{R}^N,~ t>0, \\ u(0, x)=u_0(x), & x\in \mathbb{R}^N, \\ v(0, x)=v_0(x), & x\in \mathbb{R}^N,\end{cases}
\end{equation}
where $N \ge 1, p, q \ge 1, pq>1$ and $u_0, v_0$ are bounded nonnegative function on $\mathbb{R}^N$.
They obtained the following results:
\begin{description}
	\item[i)] If $2\max\{p+1, q+1\} \ge N(pq-1)$, then no nonnegative global solution exists for any nontrivial initial data;
    
\item[ii)] If $2 \max \{p+1, q+1\}<N(pq-1)$, the solution of \eqref{1.2} blows up at finite time provided that there exist constants $C>0$ and $\alpha$ such that $u_0(x)\ge Ce^{-\alpha|x|^{2}}$;
	
\item[iii)] If $2\max\{p+1, q+1\}<N(pq-1)$, $u_0(x) \in L^{\infty}\left(\mathbb{R}^N\right) \cap L^{\alpha_1}\left(\mathbb{R}^N\right), v_0(x) \in L^{\infty}\left(\mathbb{R}^N\right) \cap$ $L^{\alpha_2}\left(\mathbb{R}^N\right)$ with $\alpha_1=(N / 2)((p q-1) /(q+1)),~\alpha_2=(N / 2)((p q-1) /(p+1))$. Then there exists $\epsilon>0$ such that if $\left\|u_0\right\|_{L^{\alpha_1}(\mathbb{R}^N)}+\left\|v_0\right\|_{L^{\alpha_2}(\mathbb{R}^N)} \le \epsilon$, every solution of \eqref{1.2} is global.
\end{description}

Recently, increasing efforts have been devoted to the existence and nonexistence of global solutions of the Cauchy problem on various space, such as manifolds (\cite{W1},\cite{ZQ}), graphs (\cite{HuWang},\cite{Len}), metric measure space(\cite{FHS}). In particular, Lin-Wu \cite{LWcv} studied the following Cauchy problem 
\begin{equation}\label{1.4}
	\left\{\begin{array}{lll}
		u_t=\Delta u+u^p &\text{ in }& V \times(0,+\infty), \\[1mm]
		u(x, 0)=\phi(x)\ge,\not\equiv0\;\; &\text{ on }& V,
	\end{array}\right.
\end{equation}
%defined by
%\begin{equation}\label{1.5}
	%\Delta u(x)=\frac1{\mu(x)}\sum_{y\in V}\omega_{x y}(u(y)-u(x)),~x\in V.
%\end{equation}
where $G=(V, E)$ is a locally finite connected graph satisfying the condition ${CDE}^{\prime}(n, 0)$ and polynomial volume growth of degree $m$, $\Delta$ is the usual graph Laplace operator on $G$ and $\phi$ is a bounded function on $V$. They discovered that if $1+\frac{2}{m}<p$, there admits a global solution to \eqref{1.4} provided that $\phi$ is sufficiently small, and if $1<p<\frac{2}{m}+1$, then any solution of \eqref{1.4} blows up in finite time.

In this paper, we study the following discrete weakly coupled system
\begin{equation}\label{1.4.}
\left\{\begin{array}{ll}
u_t=\Delta u+v^p, &t>0, \quad x\in V, \\
v_t=\Delta v+u^q,  &t>0, \quad x\in V, \\
u(0, x)=u_0(x), &x\in V, \\
v(0, x)=v_0(x), &x\in V,
\end{array}\right.
\end{equation}
where $pq>1$, $p,q>0$, $G=(V, E, \omega, \mu)$ is a locally finite connected graph satisfying $D_{\mu}<\infty$, $\Delta$ is the usual graph Laplacian on $G$ (cf. \eqref{2.1.}) and $u_0, v_0$ are bounded nonnegative functions on $V$. The problem \eqref{1.4.} is a spatial-discrete version of \eqref{1.2}. It also comes directly from mathematical biological models and chemical reactor theory (cf. \cite{Fife,SK, SS, Wang}), and the results for \eqref{1.4.} could be used to the study of these equations as shown by Aronson-Weinberger \cite{AW}.

Recently, equations and analysis on graphs have attracted considerable attention, see \cite{BHua,BS,CKKLLP,GX,HJL,HuangKW,HuangKM,HLLY,HuWang,KL,Len,LZ,Sun, RW} and references therein. Bauer-Horn-Lin-Lippner-Mangoubi-Yau \cite{BHLLMY} established the Gaussian upper bound for a graph satisfying ${CDE}(n, 0)$. Later, Horn-Lin-Liu-Yau \cite{HLLY} obtained the Gaussian lower bound for a graph satisfying ${CDE}^{\prime}(n, 0)$.

Based on these heat kernel estimates, Wu \cite{Wu} proved that
\begin{theo}\label{t2.3.}
Let $G=(V,E,\omega,\mu)$ be a graph. For any $m\in\mathbb{R}^{+}$. Suppose $p,q\ge 1,$~$pq>1$, $\mu_{\min}>0,~ \mu_{\max}<\infty,~ \omega_{\min}>0~\text{ and } ~D_{\mu}<\infty.$  
Assume $G$ satisfies the curvature dimension condition $CDE'(n,0)$. 
\begin{description}
    \item[(i)] 
    If $\frac{\max\{p,q\}+1}{pq-1}\ge\frac{m}{2}$ and $G$ satisfies the volume growth condition:
    There exist positive constants $m$, $c_{1}$ and $c_{2}$ such that $c_1r^{m}\le\mathcal{V}(x,r)\le {c}_{2}r^{m}$ for all $x\in V$ and $r>0$. Then all nontrivial nonnegative solution of \eqref{1.4.} blows up in finite time.
    \item[(ii)]
    Fix $x_0\in V$. Suppose $G$ satisfies the volume growth condition $\mathcal{V}(x_0,R)\le CR^{m}$, $\forall R>0$ for some constant $C>0$. Assume that there exist $t_2>0$ and sufficiently large $C_{11}>0$ such that 
    $u_{0}(x)\ge C_{11}P(t_2,x_0,x)\text{ and }v_0(x)\ge C_{11}P(t_2,x_0,x)~\text{ for all }~x\in V.$
    Then all nonnegative solutions of \eqref{1.4.} will blow up at finite time $T$.
\end{description}    
\end{theo}

However, \cite{Wu} leaves two unresolved issues:
\begin{itemize}
\item  Existence of Solutions: The analysis only addresses blowup scenarios but does not establish the existence of solutions even locally for short time intervals  $t\in [0, T]$.
\item Subcritical Case: The behavior of solutions when the critical exponent satisfies  $\frac{\max \{p, q\}+1}{p q-1}<\frac{m}{2}$  remains entirely unaddressed.
\end{itemize}
These questions are pivotal for a complete characterization of the blowup vs. global existence threshold for the semilinear parabolic system \eqref{1.4.}, yet they have not been tackled in the discrete geometric setting.

Resolving these unaddressed issues is essential for a rigorous theory of \eqref{1.4.}:
\begin{itemize}
\item Existence Theory: Existence is a prerequisite for any parabolic problem; without it, the blowup analysis lacks a foundational basis.
\item Subcritical Analysis: The case  $\frac{\max \{p, q\}+1}{p q-1}<\frac{m}{2}$  requires delicate balance between nonlinearity and diffusion, necessitating novel techniques that link heat kernel decay with geometric properties of the graph.
\end{itemize}

In this paper, we address these challenges via the  $L^p\text{-}L^q$  heat kernel estimates developed in Theorem 2.2:
\begin{itemize}
\item Local Existence: Using a contraction mapping principle, we first prove the existence of nonnegative solutions  $(u, v)$  to \eqref{1.4.} on a short time interval  $t \in [0, T]$  (Theorem 2.4), filling the existence gap.

\item Global Existence for Small Initial Data: Employing the heat kernel’s temporal decay rate $t^{-m / 2(1 / p-1 / q)}$  and the polynomial volume growth dimension  $m$ , we establish decay estimates for solutions in two specific  $\ell^p(V)$  and  $\ell^q(V)$  spaces (Theorem \ref{t2.5.}). Specifically, if  $p, q \geq 1 ,  r_{2}:=\frac{p q-1}{\max \{p, q\}+1} \frac{m}{2}>1$ , and the initial data satisfy
$\left\|u_{0}\right\|_{\ell^{r_{2}}(V)}+\left\|v_{0}\right\|_{\ell^{r_{2}}(V)}^{p} \leq \hat{C}_{1} \quad(\hat{C}_{1} \text{ is a positive constant})$,
then $(u,v)$  exists globally and remains bounded for all  $t>0$.
\end{itemize}
Notably, our Theorem \ref{t2.5.} complements Theorem 1.1 by establishing global existence precisely in the subcritical case  $\frac{\max \{p, q\}+1}{p q-1}<\frac{m}{2}$,  contrasting with the supercritical blowup case in Theorem 1.1.

Together, these results characterize the full Fujita-type threshold for (1.4) on graphs: blowup occurs above the critical exponent, while global existence holds below it. This provides the first complete threshold analysis for semilinear parabolic systems in discrete geometric settings.

The rest of the paper is arranged in the following way. In section 2, we introduce some notions on graph and state our main results. In section 3, we present some basic results that will be used frequently in the following.  Sections 4-6 are the main part of the paper, which is devoted to the proof of the main results. In section 4, we establish the $L^{p}$-$L^{q}$ estimates of the heat operators on $G$ (Theorems \ref{t2.1}-\ref{t2.3}). Section 5 is dedicated to the proof of the existence of the local solutions of \eqref{1.4.} (Theorem \ref{t2.2}). Sections 6 is devoted to establish the existence of global solutions to the problem \eqref{1.4.} (Theorem \ref{t2.5.}).
\section{Preliminary results}

Let $G=(V,E)$ be an infinite graph with vertex set $V$ and edge set $E$.  For $x,y\in V$, let $xy$ be the edge from $x$ to $y$. We write $y\sim x$ if $xy\in E$. Let $\omega: V \times V\to[0,\infty)$ be an edge weight function satisfying $\omega_{xy}=\omega_{yx}$ for all $x,y \in V$
and $\omega_{xy}>0$ iff $x\sim y$. For each point $x\in V$, we denote 
$${\rm deg}(x)=\#\{y \in V:\, x\sim y\},$$ as the degree of $x$. We say a graph $G=(V,E)$ is locally finite if deg$(x)$ is a finite number for each point $x\in V$. Let $\mu: V\to (0,\infty)$ be a positive measure. We also write the graph $G$ as a quadruple $G=(V,E,\omega,\mu)$. A graph is called connected if any two vertices can be connected via finite edges. Throughout this paper, we suppose that $G=(V,E,\omega,\mu)$ is a locally finite connected graph without loops or without multiple edges.

Denote the space of real-valued functions on $V$ by $V^{\mathbb{R}}$. For any $g\in V^{\mathbb{R}}$, define an integral of $g$ over $V^{\mathbb{R}}$ by
\begin{equation*}
    \int_V g{\rm d}\mu=\sum_{x\in V}g(x)\mu(x).
\end{equation*}

Define
$$
\mu_{\min }:=\inf_{x\in V}\mu(x),~\mu_{\max }:=\sup_{x\in V}\mu(x),~\omega_{\min}=\inf\limits_{e\in E} \omega_{e}~\text{ and }~ D_\mu:=\sup_{x\in V}\frac{m(x)}{\mu(x)},$$
where
$$m(x):=\sum\limits_{y \in V:\,y_{\sim} x}\omega_{xy},\;\,x\in V.$$

Denote the distance $d(x,y)$ as the smallest number of edges of a path between two vertices $x$ and $y$. We define balls
\begin{equation*}
    B(x,r):=\{y\in V:d(x,y)\le R\}.
\end{equation*}
We use the notation 
\begin{equation*}
    \mathcal{V}(A):=\sum_{x\in A}\mu(x).
\end{equation*}
to denote the volume of a subset $A\subset V$.
We usually write $\mathcal{V}(B(x,R))$ for $\mathcal{V}(x,R)$.

\subsection{The Laplacian on graphs}

Define the set of $\ell^p$ integrable functions on $V$ with respect to
the measure $\mu$ by
$$\ell^p(V)=\ell^p(V,\mu)=\left\{f\in V^{\mathbb{R}}:\sum_{x \in V} \mu(x)|f(x)|^p<\infty\right\}, \;\; 1\leq p<\infty,$$
with the norm
$$\|f\|_{\ell^p(V)}=\left(\sum_{x \in V} \mu(x)|f(x)|^p\right)^{1/p},\;\;\;f\in \ell^p(V).$$
For $p=\infty$, we define
$$\ell^{\infty}{(V)}=\left\{f\in V^{\mathbb{R}}:\sup_{x\in V}|f(x)|<\infty\right\},$$
with the norm
$$\|f\|_{\ell^{\infty}{(V)}}=\sup_{x\in V} |f(x)|,\;\;\;f\in \ell^{\infty}{(V)}.$$

Let $X$ be a real Banach space, with norm $\|~\|_X$.
\begin{definition} 
The space $C([0, T] ; X)$ consists of all continuous functions $u:[0, T] \rightarrow X$ with
\begin{equation*}
\|u\|_{C([0, T] ; X)}:=\max_{0 \le t \le T}\|u(t, \cdot)\|_X<\infty.
\end{equation*}
\end{definition}
\begin{lem}[\cite{LCE}]\label{t2.1}
 A strongly measurable function $f:[0, T] \rightarrow X$ is summable if and only if $t \mapsto\|f(t)\|_{X}$ is summable. In this case
   $$
   \left\|\int_0^T f(t)dt\right\|_{X} \leq \int_0^T\|f(t)\|_{X}dt.
   $$
\end{lem}

The usual graph Laplacian $\Delta$ is defined by
\begin{equation}\label{2.1.}
    \Delta H(x)=\frac1{\mu(x)}\sum_{y\in V:\, y\sim x}\omega_{x y}(H(y)-H(x)),~x\in V,~H\in  V^{\mathbb{R}}.
\end{equation}
Obviously, the measure $\mu$ has a crucial part in the definition of Laplacian. Considering the weight $\omega$ on $E$, two typical forms of the Laplacian as follows:
\begin{itemize}
    \item $\mu(x)=\sum\limits_{y\sim x}\omega_{xy}$ for all $x\in V$, which is the normalized graph Laplacian;
    \item $\mu(x)\equiv 1$ for all $x\in V$, which is called the combinatorial graph Laplacian.
\end{itemize}

It follows from \cite{KL} that
$\Delta$ is bounded on $\ell^p(V,\mu)$ for all $p\in [1,\infty]$ if and only if
$$D_\mu<\infty.$$

In this paper, we always make the assumption that the measure $m$ on V is non-degenerate, namely 
$$\mu_{min}=\inf_{x\in V}\mu(x)>0.$$ 
Such an assumption is quite mild, for it is automatically satisfied for any combinatorial Laplacian, and it gives rise to extremely useful properties for $\ell^p$ spaces.
\begin{lem}[\cite{HL}]\label{l3.4}
    Suppose $p_0\ge 1$ and $\mu_{min}>0$. Then $\ell^{p_0}(V)\subset\ell^{\infty}(V)$. Furthermore, $$\|f\|_{\ell^{\infty}(V)} \le\mu_{min}^{-\frac{1}{p_0}}\|f\|_{\ell^{p_0}(V).}$$
\end{lem}

\begin{lem}[\cite{HL}]\label{l3.5}
    Let $1\le p_0\le p_1$. Assume $\mu_{min}>0$. Then $\ell^{p_0}(V)\subset\ell^{p_1}(V)$. Moreover, $$\|f\|_{\ell^{p_{1}}(V)}\le\mu_{min}^{\frac{p_0-p_1}{p_0p_1}}\|f\|_{\ell^{p_{0}}(V)}.$$
\end{lem}

Now, we introduce the gradient forms associated to the Laplacian on graphs as introduced in \cite{BHLLMY}.

 For $f,g\in V^{\mathbb{R}}$, the gradient form of $f$ and $g$ reads
\begin{equation*}
    \Gamma(f,g)(x)=\frac1{2\mu(x)}\sum_{y\in V:\,y\sim x}{\omega}_{x y}(f(y)-f(x))(g(y)-g(x)),
\end{equation*}
the iterated gradient form $\Gamma_2$ is defined by
\begin{equation*}
    \Gamma_2(f,g)(x)=\frac12(\Delta \Gamma(f,g)-\Gamma(f,\Delta g)-\Gamma(\Delta f,g))(x).
\end{equation*}
We will write $\Gamma(f)=\Gamma(f,f)$ and $\Gamma_2(f)=\Gamma_2(f,f)$.

In order to prove stochastic completeness for the semigroups associated to unbounded Laplacian on graphs, Hua and Lin \cite{HL} introduce a condition for the completeness of infinite weighted graphs: 
A graph $G=(V, E, \omega, \mu)$ is called {\it complete} if there exists a nondecreasing sequence of finitely supported functions $\{\eta_k\}_{k = 1}^{\infty}$ on $V$ such that
$$\lim_{k \to \infty} \eta_k =\mathbf{ 1} \quad \text{and} \quad \Gamma(\eta_k) \leq \frac{1}{k},$$
where $\mathbf{1}$ is the constant function on $V$. 
This condition was defined for Markov diffusion semigroups in Definition 3.3.9 of \cite{BGL} and adopted to graphs in \cite{GLLY,HL}. It follows from Theorem 2.8 in \cite{HL} that a large class of graphs is complete.

\subsection{The heat kernel on graphs and curvature dimension condition}\label{s2.2}

For $t > 0$ and $x,y\in V$, if $P(t,x,y)$ is the smallest nonnegative function satisfying
\begin{equation*}
\begin{cases}
    \frac{\partial}{\partial t}P(t,x,y)+\Delta P(t,x,y)=0 \quad\text{in either }x\text{ or }y,~\text{ for } t>0,\ x,y\in V,\\
    P(0,x,y)=\delta_x(y)=
    \begin{cases}
        \frac{1}{\mu(y)},&x = y,\\
        0,&x\neq y,
    \end{cases}
\end{cases}
\end{equation*}
then we say that $P$ is the heat kernel on $G$.

For convenience, we summarize some important properties of the heat kernel $P(t,x,y)$ on $G$ in the following propositions.

\begin{prop}[\cite{RW}]\label{p2.1} For $t,s>0$ and any $x,y\in V$, we have
    
    {\rm(i)}~$P_t(t,x,y)=\Delta_xP(t,x,y)=\Delta_yP(t,x,y)$,\vspace{1mm}
    
    {\rm(ii)}~$P(t,x,y)> 0$,\vspace{1mm}
    
    {\rm(iii)}~$P(t,x,y)=P(t,y,x)$,\vspace{1mm}
    
    {\rm(iv)}~$\sum\limits_{y\in V}P(t,x,y)\mu(y)\leq 1$,\vspace{1mm}
    
    {\rm(v)}~$\sum\limits_{z\in V}P(t,x,z)P(s,z,y)\mu(z)=P(t+s,x,y)$.
\end{prop}

\begin{prop}[\cite{BHLLMY}]\label{p2.2}
    Suppose $D_{\mu}<\infty$. Then for any $u\in \ell^{\infty}(V)$,
    \begin{equation}\label{2.3}
        \begin{aligned}
            \sum_{y \in V} P(t, x, y) u(y) \mu(y) & =\sum_{k=0}^{+\infty} \frac{t^k \Delta^k}{k !} u(x) \\
            & :=u(x)+t \Delta u(x)+\frac{t^2}{2 !} \Delta^2 u(x)+\cdots
        \end{aligned}
    \end{equation}
    for all $t\in[0,+\infty)$ and $x\in V$.
\end{prop}

Let us recall the curvature dimension conditions introduced in \cite{BHLLMY}.

\begin{definition}
    Let $x\in V$, $n\in\mathbb{R}^+=(0,\infty)$ and  $K\in\mathbb{R}$. We call that a graph $G$ satisfies the exponential curvature dimension inequality $CDE(x,n,K)$ at the vertex $x$, if for any function $f:V\to \mathbb{R}^+$ satisfying $\Delta f(x)<0$, there holds:
    $$\Gamma_2(f)(x)-\Gamma\Bigg(f,\frac{\Gamma(f)}{f}\Bigg)(x)\geq\frac1{n}[(\Delta f)(x)]^2+K\Gamma(f)(x).$$
    We say that $CDE(n,K)$ is satisfied if $CDE(x,n,K)$ is satisfied for all $x\in V$.
\end{definition}

\begin{definition} Let $x\in V$, $n\in\mathbb{R}^+$ and  $K\in\mathbb{R}$. A graph $G$ satisfies the exponential curvature dimension inequality $CDE'(x,n,K)$, if for any function $f:V\to \mathbb{R}^+$, we have
    $$\Gamma_2(f)(x)-\Gamma\Bigg(f,\frac{\Gamma(f)}{f}\Bigg)(x)\geq
    \frac1{n}[f(x)(\Delta \log f)(x)]^2+K\Gamma(f)(x),~x\in V.$$
    We say that $CDE'(n,K)$ is satisfied if $CDE'(x,n,K)$ is satisfied for all $x\in V$.
\end{definition}

The relation between $CDE(n, K)$ and $C D E^{\prime}(n, K)$ is the following:

\begin{remark}[\cite{BHLLMY, HLLY}]\label{r2.1}   $CDE^{\prime}(n,K)$ implies $CDE(n,K)$. \end{remark}

\begin{prop}[\cite{BHLLMY}]\label{p2.3} 
Let $G=(V, E, \omega, \mu)$ be a graph.
Assume $\mu_{\min}>0, \mu_{\max}<\infty, \omega_{\min}>0$, $D_{\mu}<\infty$ and $G$ satisfies $CDE(n,0)$.  Then there exists a positive constant $C_1=C_1(n, \mu_{max}, \omega_{\min})$ such that, for any $x, y \in V$ and $t> 0$,
    \begin{equation*}
        P(t,x,y)\leq\frac{C_1}{\mathcal{V}(x,\sqrt{t})}.
    \end{equation*}
\end{prop}

\begin{prop}[\cite{GLLY}]\label{p2.4}
Let  $G=(V, E, w, \mu)$  be a complete graph satisfying  $CDE^{'}(n,0)$.
Assume  $\mu_{\min}>0$,  $\mu_{\max}<\infty$ and  $\omega_{\min}>0$.
Then there exists a constant  $C^*=C^*(n,\mu_{max}, \omega_{\min})>0$ so that for any  $x,y\in V$  and for all  $t>0$,
\begin{equation}
P(t,x,y)\leq\frac{C^*}{\mathcal{V}(x,\sqrt{t})}.
\end{equation}
\end{prop}
\subsection{Li-Yau inequality and Hanack inequality on graphs}

We introduce the following Li-Yau inequality on graphs.

\begin{prop}[\cite{BHLLMY}]
Suppose $G=(V, E, \omega, \mu)$ is a graph satisfying $CDE(n,0)$. Assume 
$$\mu_{\min}>0, \mu_{\max}<\infty, \omega_{\min}>0 \text{ and } D_{\mu}<\infty.$$ 
If $u$ is a positive solution to the heat equation on $G$, then
\begin{equation}
\frac{\Gamma(\sqrt{u})}{u}(t,x)-\frac{\partial_t\sqrt{u}}{\sqrt{u}}(t,x)\leq \frac{n}{2t}
\end{equation}
for all $t > 0$ and $x\in V$. 
\end{prop}

The following Hanack inequality on graphs is the direct consequence of the above Li-Yau inequality.
\begin{prop}[\cite{BHLLMY}]\label{p2.6}
 Let $G=(V, E, \omega, \mu)$ be a graph. Assume 
$$\mu_{\min}>0, \mu_{\max}<\infty, \omega_{\min}>0 \text{ and } D_{\mu}<\infty,$$ 
and $G$ satisfies $CDE(n,0)$.
    If $u$ is a positive solution to the heat equation on $G$, then
 \begin{equation}
 u(T_1,x)\leq u(T_2,y)\left(\frac{T_2}{T_1}\right)^{n}\exp\left(4\frac{\mu_{\max}}{\omega_{\min}}\frac{d^2(x,y)}{T_2 - T_1}\right)
    \end{equation}
 for all $x,y\in V$  and  any $T_1<T_2$.
\end{prop}

The subsequent result provides a version of the Harnack inequality upon removing the  $D_{\mu}<\infty$  condition. Notably, in this scenario,  G  is strengthened to be a complete graph, and the curvature condition is elevated from  $CDE(n,0)$  to  $CDE'(n,0)$.

\begin{prop}[\cite{GLLY}]\label{p2.7}
Assume $G=(V,E,\omega,\mu)$ is a complete graph and $f(t,\cdot)\in\ell^{p}(V)$ with $p\in[1,\infty]$ for every $t\ge 0$. Suppose $\mu_{\min}>0,~ \mu_{\max}<\infty, ~\omega_{\min}>0,~$
$G$ satisfies $CDE^{'}(n,0)$ and $f$ is a positive solution to the heat equation on $G$.
Then for all $x,z\in V$ and any $t<s$, one has 
\begin{equation}
f(t,x)\leq f(s,z) \left(\frac{s}{t}\right)^{n}\exp\left(4\frac{\mu_{\max}}{\omega_{\min}}\frac{d^2(x,z)}{s- t}\right).
\end{equation}
\end{prop}

Let $\mathbb{Z}^{+}$ be the collection of all positive integers.

\subsection{Main results}

 We now state main results.
 \begin{theo}\label{l3.6}
Let $G=(V, E, \omega, \mu)$ be a graph. Suppose 
     $$\mu_{\min}>0, \mu_{\max}<\infty, \omega_{\min}>0 \text{ and } D_{\mu}<\infty.$$ 
     Let $r\in[1,+\infty]$. Assume $G$ satisfies $CDE(n,0)$ and the lower volume growth condition: 
     \begin{description}
         \item[\bf{(LVG$m$)}] There exist positive constants $c_3$ and $m$ so that $\mathcal{V}(x,R)\ge c_{3}R^{m}$ 
         for some $m>0$, all $x\in V$ and all $R>0$.
     \end{description}
 Then 
\begin{equation}\label{2.7}
\left\|P\left(t,x,\cdot \right)\right\|_{\ell^r(V)} \le C_4 t^{-\frac{m}{2}\left(1-\frac{1}{r}\right)}~\text{for}~x\in V,~t>0,
\end{equation}
     where $C_4=C_{4}(n, \mu_{max}, \omega_{\min},c_3,r,m)$ is a positive constant.
 \end{theo}
 
\begin{theo}\label{t2.1.}
Let $G=(V, E, \omega, \mu)$ be a graph and $1\le b, a, r \le \infty$ satisfy the relation:
    \begin{equation}\label{3.4}
        1+\frac{1}{b}=\frac{1}{a}+\frac{1}{r}.
    \end{equation} 
Assume 
$$g\in \ell^{a}(V),~ \mu_{\min}>0,~ \mu_{\max}<\infty,~ \omega_{\min}>0,~ D_{\mu}<\infty,$$
and $G$ satisfies $CDE(n,0)$ and {\bf{(LVG$m$)}}.
Then  
    \begin{equation}\label{2.9}
        \left\|\sum_{y\in V}P\left(t,\cdot,y\right)g(y)\mu(y)\right\|_{\ell^b(V)}\le \hat{C}_4\|g\|_{\ell^a(V)}t^{-\frac{m}{2}\left(1-\frac{1}{r}\right)}
    \end{equation} 
    for all $t> 0$ and $x\in V$, where 
    \bess
   \hat{C}_4 =\begin{cases}
       C_4,~&b\in(1,+\infty],\\
       1,~&b=1.
    \end{cases}\eess
\end{theo}

\begin{theo}\label{t2.3}
Let $G=(V, E, \omega, \mu)$ be a complete graph and $1\le b, a, r \le \infty$ satisfy the relation \eqref{3.4}.  
Assume 
$$g\in \ell^{a}(V),~ \mu_{\min}>0,~ \mu_{\max}<\infty,~ \omega_{\min}>0$$
and $G$ satisfies $CDE^{'}(n,0)$ and {\bf{(LVG$m$)}}.
Then there exists constant $\tilde{C}_{4}>0$ depending on $n, \mu_{max}, \omega_{\min},c_3,r,m$ so that 
\begin{equation*}
    \left\|P\left(t,x,\cdot \right)\right\|_{\ell^r(V)} \le \tilde{C}_4 t^{-\frac{m}{2}\left(1-\frac{1}{r}\right)}
\end{equation*}
and
\begin{equation*}
    \left\|\sum_{y\in V}P\left(t,\cdot,y\right)g(y)\mu(y)\right\|_{\ell^b(V)}\le \tilde{C}_4\|g\|_{\ell^a(V)}t^{-\frac{m}{2}\left(1-\frac{1}{r}\right)}
\end{equation*} 
for all $t> 0$ and $x\in V$.
\end{theo}

\begin{definition}
Let $T>0$. Two functions $u=u(t, x), v=v(t, x)$ in $[0, T] \times V$ is called a solution of \eqref{1.4.} in $[0, T]$ if  $(u,v)$  satisfies

1) $u(t,x),v(t,x)\in C\left( [0,T]; \ell^{\infty}(V)\right)$;

2) $u_t=\Delta u+v^p,~ v_t=\Delta v+u^q \text{ for every } t \in[0, T] \text{ and } x \in V.$
	
\end{definition}

\begin{definition}
Assume $(u,v)$ is a solution to \eqref{1.4.} in $[0,T)$. If $u^{2}(\tilde{t},\tilde{x})+v^{2}(\tilde{t},\tilde{x})\not=0$ for some $\tilde{t}\in(0,T)$ and $\tilde{x}\in V$, then we call $(u,v)$ is a nontrivial solution to \eqref{1.4.}.
\end{definition}

\begin{theo}\label{t2.2}
	Assume $p,q> 0$, $pq>1$ and $D_{\mu}<\infty$. For any given $u_0$ and $v_0$ satisfying $u_0(x)$, $v_0(x)$ are nonnegative and bounded on $V$, there exists $T$ with $0<T<+\infty$ so that \eqref{1.4.} admits a nonnegative solution $(u,v)$ for $t\in [0,T]$.
\end{theo}

\begin{theo}\label{t2.5.}
    For any $m\in(0,+\infty)$. Suppose $G=(V,E,\omega,\mu)$ is a graph satisfying CDE$(n,0)$ and the condition {\bf (LVG$m$)}.
    Assume
     $$ \mu_{\min}>0,~ \mu_{\max}<\infty,~ \omega_{\min}>0~\text{ and }~ D_{\mu}<\infty~ .$$
Let $p,q\ge 1$, $pq>1$,  $T>0$ and $r_2=\frac{m}{2}\left(\frac{pq-1}{\max\{p,q\}+1}\right)$. 
Assume $$\frac{\max\{p,q\}+1}{pq-1}<\frac{m}{2},$$ and $$\left\|u_0\right\|_{\ell^{r_{2}}(V)}+\left\|v_0\right\|^{p}_{\ell^{r_2}(V)}\le\hat{C}_{1}$$ for some $\hat{C}_{1}>0$.  If $(u,v)$ is a solution to \eqref{1.4.} for $t\in[0,T]$, then $(u,v)$ is a global solution to \eqref{1.4.}.
\end{theo}

\begin{remark}
	Our results could apply to the lattice $\mathbb{Z}^{m}$.
\end{remark}

\section{Auxiliary result}
\subsection{Some auxiliary inequalities on graphs}
Let $G=(V,E,\omega,\mu)$ be a graph. In our quest to prove Theorems \ref{l3.6} and \ref{2.3}, we must first establish certain fundamental inequalities on the graph $G = (V, E, \omega, \mu)$. Specifically, we commence with the following H\"{o}lder-type inequalities, which play a crucial role in the subsequent derivations.

\begin{lem}\label{l3.2}
Assume $1\le a,b \le \infty$, $\frac1{a}+\frac1{b}=1$, $f\in \ell^{a}(V)$, $g\in \ell^{b}(V)$. Then 
\begin{equation}\label{3.1}
\|fg\|_{\ell^{1}(V)} \le\|f\|_{\ell^a(V)}\|g\|_{\ell^b(V)}
\end{equation}
\end{lem}
\begin{proof}
If $\|f\|_{\ell^a(V)}\|g\|_{\ell^b(V)}=0$, then $f\equiv 0$ or $g\equiv 0$ on $V$. Thus \eqref{3.1} follows immediately. Next, we suppose that 
	$$\|f\|_{\ell^a(V)}\not=0\text{~and~}\|g\|_{\ell^b(V)}\not=0.$$
To prove \eqref{3.1}, it suffices to prove that
	\begin{equation*}
\left\|\frac{f}{\|f\|_{\ell^a(V)}} \frac{g}{\|g\|_{\ell^b(V)}}\right\|_{\ell^{1}(V)}\le 1 .
	\end{equation*}
	Hence we can assume that $$\|f\|_{\ell^a(V)}=1\text{~and~}\|g\|_{\ell^b(V)}=1.$$ By Young's inequality (see last line of page 708 of \cite{LCE}), we see that for $1<a,b<\infty$,
	\begin{equation*}
		\begin{aligned}
			\sum_{x\in V}|f(x)g(x)|\mu(x)&=\sum_{x\in V}|f(x)g(x)|\mu^{\frac1{a}+\frac1{b}}(x) \\
			 &\le\sum_{x\in V} \left[\frac{|f(x)\mu^{\frac1{a}}(x)|^a}{a}+\frac{|g(x)\mu^{\frac1{b}}(x)|^b}{b}\right] \\
			&=\frac{1}{a}\|f\|_{{\ell}^a(V)}^a+\frac{1}{b}\|g\|_{\ell^b(V)}^b \\
			&=\frac{1}{a}+\frac{1}{b}\\
			&=1.
		\end{aligned}
	\end{equation*}
When either $a =\infty$ or $b=\infty$, without loss of generality, assume $a = \infty$. Then, by the properties of the norm and the nature of the sum $\sum_{x\in V}|f(x)g(x)|\mu(x)$, it is straightforward to observe that $\|fg\|_{\ell^1(V)}\leq 1$ holds trivially.
Therefore, we get $\|f g\|_{\ell^{1}(V)} \le 1$.

We now complete the proof. 
\end{proof}

The following result follows from Dirichlet form theory in \cite{KL} and \cite{FOT}.
\begin{lem}\label{l3.9}
Let $k\in [1,+\infty]$. Suppose $I\in\ell^{k}(V)$. Then 
\bess
\sum_{y\in V} P(t,\cdot,y)I(y)\mu(y)\in\ell^{k}(V)~\text{ for all }~t\ge 0.
\eess
Moreover,
\bess\left\|\sum_{y\in V} P(t,\cdot,y)I(y)\mu(y)\right\|_{\ell^{k}(V)}\le\left\|I\right\|_{\ell^{k}(V)}~\text{ for }~t\ge 0.\eess
\end{lem}

\subsection{Comparison Principle}
In this subsection, we study a comparison principle for later use.

\begin{lem}\label{l3.10}
Suppose $D_{\mu}<\infty$, $q\ge 1, T>0, u_1, u_2, v_1, v_2 \in C\left([0,T];\ell^{\infty}(V)\right)$, $u_{1}$ and $u_{2}$ are nonnegative, $g$ is nondecreasing and globally Lipschitz continuous and $u_{1,0}, u_{2,0}, v_{1,0}, v_{2,0} \in V^{\mathbb{R}}$.
Assume that $u_1, v_1, u_2, v_2$ satisfy
\bes
&&u_1(t, x) \le \sum_{y \in V} P(t, x, y) u_{1,0}(y)\mu(y)+\int_0^t \sum_{y \in V} P(t-s, x, y) g\left(v_1\right)(s,y)\mu(y)ds,\label{3.12}\\
&&v_1(t, x) \le \sum_{y \in V} P(t, x, y) v_{1,0}(y) \mu(y)+\int_0^t \sum_{y \in V} P(t-s, x, y) u_1^q(s,y)\mu(y)ds,\label{3.13}\\
&&u_2(t, x) \ge \sum_{y \in V} P(t, x, y) u_{2,0}(y) \mu(y)+\int_0^t \sum_{y \in V} P(t-s, x, y) g\left(v_2\right)(s,y)\mu(y)ds,\label{3.14}\\
&&v_2(t, x) \ge \sum_{y \in V} P(t, x, y) v_{2,0}(y) \mu(y)+\int_0^t \sum_{y \in V} P(t-s,x,y)u_2^q(s,y)\mu(y)ds\label{3.15}
\ees
for $t>0$ and $x \in V$,
and
\bes
&&u_1(0, x)=u_{1,0}(x) \le u_{2,0}(x)=u_2(0,x),x\in V,\label{3.16}\\
&&v_1(0, x)=v_{1,0}(x) \le v_{2,0}(x)=v_2(0,x),x\in V.\label{3.17.}
\ees
Then $u_1(t, x) \le u_2(t, x)$ and $v_1(t, x) \le v_2(t, x)$ for $t \in[0, T]$ and $x \in V$. 
\end{lem}

\begin{proof}
Clearly, $v_1$ and $v_2$ are bounded. Substracting \eqref{3.14} from \eqref{3.12}, since $g$ is nondecreasing and locally Lipschitz continuous, by \eqref{3.16}, we can find constant $L_{1}>0$ so that 
\bes
(u_1-u_2)(t,x)&\le& \int_0^t \sum_{y \in V} P(t, x, y)\left[u_{1,0}(y)-u_{2,0}(y)\right]\mu(y)\nm\\
&&+\int_0^t \sum_{y \in V} P(t-s, x, y)\left(g\left(v_1\right)-g\left(v_2\right)\right)(s,y)\mu(y)ds\nm\\
&\le& \int_0^t \sum_{y \in V} P(t-s, x, y) L_1\left[\left(v_1-v_2\right)(s,y)\right]_{+}\mu(y)ds.\nm
\ees
It follows that \bes
[(u_1-u_2)(t,x)]_{+}\le\int_0^t \sum_{y\in V}P(t-s,x,y)L_1\left[\left(v_1-v_2\right)(s,y)\right]_{+}\mu(y)ds.\label{3.18.}
\ees
Substracting \eqref{3.15} from \eqref{3.13}, by \eqref{3.17.}, we see that
\bes
\left(v_1-v_2\right)(t,x) &\le&\int_0^t \sum_{y \in V} P(t-s, x, y)\left[u_1^q(s,y)-u_2^q(s,y)\right]\mu(y)ds\nm\\
	&\le&\int_0^t \sum_{y\in V}P(t-s,x,y)\left[u_1^q(s,y)-u_2^q(s,y)\right]_{+}\mu(y)ds,\label{3.12,}
\ees
and hence that 
\bes\label{3.20.}
\left[(v_1-v_2)(t, x)\right]_{+}&\le&\int_0^t\sum_{y\in V}P(t-s,x,y)\left[u_1^q(s,y)-u_2^q(s,y)\right]_{+}\mu(y)ds\nm\\
&\le&\int_0^t \sum_{y\in V}P(t-s,x,y)\left\|\left[u_1^q(s,\cdot)-u_2^q(s,\cdot)\right]_{+}\right\|_{\ell^{\infty}(V)}\mu(y)ds.
\ees
If $q>1$, by the mean value theorem, we have for any $(s, y)$
satisfying $u_{2}(s, y) \neq u_{1}(s, y)$,
\bess
\left|\frac{u_1^q(s,y)-u_2^q(s, y)}{u_1(s, y)-u_2(s, y)}\right|=\left|q \theta^{q-1}(s,y)\right| 
&\le&q\left(\left\|u_1\right\|_{C\left([0,T];\ell^{\infty}(V)\right)}+\left\|u_2\right\|_{C\left([0,T];\ell^{\infty}(V)\right)}\right)^{q-1}\nm\\
&=:&L_2,
\eess
where $\theta(s, y)$ is between $u_{1}(s, y)$ and $u_{2}(s, y)$.
It follows that for $q\ge 1$,
$$
\left|u_1^q(s,y)-u_2^q(s,y)\right|\le L\left|u_1(s, y)-u_2(s, y)\right|
$$
for all $s \in[0, t]$ and $y \in V$, where $L=L_2+1.$ Due to $f(x)=x^q$ is nondecreasing in $x \ge 0$, we conclude that
$u_1^q(s, y)-u_2^q(s, y)$ has the same sign with $u_1(s, y)-u_2(s, y)$, and hence that
$$
\left[u_1^q(s, y)-u_2^q(s, y)\right]_{+} \leq L\left[u_1(s, y)-u_2(s, y)\right]_{+} .
$$
Substituting this into \eqref{3.20.}, by Proposition \ref{p2.1} (iv), we obtain 
\bess
\left[\left(v_1-v_2\right)(t,x)\right]_{+}&\le& \int_0^t \sum_{y\in V}P(t-s,x,y)
L\left\|\left[u_1(s,)-u_2(s, \cdot)\right]_{+}\right\|_{\ell^{\infty}(V)}\mu(y)ds\nm\\
	&\le& \int_0^t L\left\|\left[u_1(s,)-u_2(s, \cdot)\right]_{+}\right\|_{\ell^{\infty}(V)}ds.
\eess
Substituting this into \eqref{3.18.}, and using Proposition \ref{p2.1} (iv) again, we have
$$\left[\left(u_1-u_2\right)(t, x)\right]_{+} \le L_1 L \int_0^t \int_0^s\left\|\left[u_1(\tau,\cdot)-u_2(\tau, \cdot)\right]_{+}\right\|_{\ell^{\infty}(V)}d\tau ds.$$
This implies that 
\bes\label{3.22..}
\left\|\left[\left(u_1-u_2\right)(t, \cdot)\right]_{+}\right\|_{\ell^{\infty}(V)}\leq L_1L\int_0^t \int_0^s\left\|\left[u_1(\tau, \cdot)-u_2(\tau, \cdot)\right]_+\right\|_{\ell^{\infty}(V)}d\tau ds.
\ees
Let 
\bes 
f(t):=\int_0^t \int_0^s\left\|\left[u_1(\tau, \cdot)-u_2(\tau, \cdot)\right]_{+}\right\|_{\ell^{\infty}(V)}d\tau ds.\label{3.22.}
\ees
Clearly, for any $0<\delta\ll 1$ and $s\in[0,T)$
\bes\label{3.24.}
&&\left| \left\lVert \left[(u_1 - u_2)(s + \delta, \cdot)\right]_+ \right\rVert_{\ell^\infty(V)}-\left\lVert \left[(u_1 - u_2)(s, \cdot)\right]_+ \right\rVert_{\ell^\infty(V)} \right|\nm \\
    &\leq& \left\| [(u_1 - u_2)(s +\delta, \cdot)]_{+} - [(u_1 - u_2)(s, \cdot)]_{+} \right\|_{\ell^\infty(V)} \nm \\
    & \leq& \left\| [(u_1 - u_2)(s +\delta, \cdot)] - [(u_1 - u_2)(s, \cdot)] \right\|_{\ell^\infty(V)}. 
\ees
In view of this, since $u_1, u_2 \in C\left([0,T];\ell^{\infty}(V)\right)$, by \eqref{3.24.},
we conclude that
$$\left\|\left[u_1(\tau, \cdot)-u_2(\tau, \cdot)\right]_{+}\right\|_{\ell^{\infty}(V)}\text{ is continuous } w.r.t.~ \tau \in[0, T).$$ The continuity at $T$ can be deduced by a similar argument. Then by virtue of \eqref{3.22.}, it is easy to check that $f \in C^2[0, T]$. Thus, from \eqref{3.22..} and \eqref{3.22.}, we have
\bes
f^{\prime\prime}(t)\leq L_1Lf(t)=:\alpha f(t)\text~{ for }~t\in[0,T],\label{3.17}
\ees
where $\alpha=L_1 L>0$ is a constant independent of $t$. 
By \eqref{3.22.},  we see that
\begin{equation}\label{3.26,}
f(0)=f^{\prime}(0)=0 \text{ and } f(t), f^{\prime}(t)\geq0 \text{ for } t\in[0,T].
\end{equation}
Due to $\alpha>0$, by \eqref{3.17} and \eqref{3.26,}, we deduce that
\[
\int_{0}^{t}f^{\prime\prime}(t)f^{\prime}(t)dt\leq\int_{0}^{t}\alpha f(t)f^{\prime}(t)dt \text{ for } t\in[0,T]
\]
and hence that
\[
\frac{(f^{\prime}(t))^{2}}{2}\leq\alpha\frac{f^{2}(t)}{2} \text{ for } t\in[0,T].
\]
By \eqref{3.26,}, this implies that
\[
f^{\prime}(t)\leq\sqrt{\alpha}f(t) \text{ for } t\in[0,T].
\]
It follows that
\[
\left[e^{-\sqrt{\alpha}t}f(t)\right]^{\prime}=e^{-\sqrt{\alpha}t}(-\sqrt{\alpha}f + f^{\prime}(t))\leq0~\text{ for }t\in[0,T].
\]
Thus we see that
\[
e^{-\sqrt{\alpha}t}f(t)\leq0 \text{ for } t\in[0,T]
\]
and hence that $f(t)\equiv 0$ by \eqref{3.26,}.
From \eqref{3.22.}, we see that $$\left\|\left[\left(u_1-u_2\right)(t,\cdot)\right]_{+}\right\|_{\ell^{\infty}(V)}\equiv0\text{ for all }t \in[0,T],$$
and hence that
$$u_1(t,x)\leq u_2(t,x)\text{ for all }t\in[0, T]\text{ and }x \in V.$$
By virtue of \eqref{3.12}, similarly, we deduce
$v_1(t, x) \leq v_2(t, x)$ for all $t \in[0, T]$ and $x \in V$.

We now complete the proof.
\end{proof}

\section{Proof of Theorems \ref{l3.6}-\ref{t2.3}}

Having laid out the essential propositions and foundational concepts in the prior discussions, we now delve into the detailed proofs of Theorems \ref{l3.6} and \ref{t2.1.}, these theorems serving as cornerstones for our subsequent analyses.

\subsection{Proofs of Theorems \ref{l3.6} and \ref{t2.1.}}
\begin{proof}[Proof of Theorem \ref{l3.6}]
    Since $\mathcal{V}(x,r)\ge c_{3}R^{m}$ for some $c_3>0$ and all $R>0$, by Proposition \ref{p2.3}, we know that 
    \begin{equation*}
        P(t,x,y)\le\frac{C_1}{\mathcal{V}(x,\sqrt{t})}\le\frac{C_1}{c_3t^{\frac{m}{2}}}~\text{for~all}~ t>0~\text{and}~x,~y\in V.
    \end{equation*}
Then for all $x\in V$, 
\begin{equation}\label{4.1}
        \|P(t,x,\cdot)\|_{\ell^{\infty}(V)}\le \tilde{C}_4t^{-\frac{m}{2}},
    \end{equation}
where $\tilde{C}_4=\frac{C_1}{c_3}$. Thus, we know that \eqref{2.7} holds when $r=\infty$. 
    
    Next, we suppose that $r\in[1,+\infty)$.
    Since $P>0$ satisfies 
    $$P_{t}(t,x,y)+\Delta_xP(t,x,y)=0,$$ 
    by Proposition \ref{p2.6},
    we know that for all $t>0$ and $x,y,z\in V$, which implies that
    $$
    P(t,y,x) \leq P(2 t, z, x) 2^n \exp \left(\frac{4 d^2(y, z) D}{t}\right)
    $$
    for all $t>0$ and $x,y,z\in V$, where 
    $$D=\frac{\mu_{max}}{\omega_{min}}.$$ 
From this and \eqref{4.1}, we see that
    \bes\label{3.5.}
    P^r(t,y,x)&\leq& P^r(2t,z,x)2^{nr}\exp\left(\frac{4rd^2(z,y)D}{t}\right)\nm\\
    &\leq& P^r(2t,z,x)2^{nr}\exp(4rD)\nm\\
    &=&P(2t,z,x)P^{r-1}(2t,z,x)2^{nr}\exp(4rD)\nm\\
    &\leq& P(2t,z,x)\left(\tilde{C}_4 (2t)^{-\frac{m}{2}}\right)^{r-1}2^{nr}\exp(4rD)\nm\\
    &=&2^{-\frac{m}{2}(r-1)}\tilde{C}_4^{r-1}2^{nr}\exp(4rD)t^{-\frac{m}{2}(r-1)}P(2t,z,x)\nm\\
    &=&C_5 t^{-\frac{m}{2}(r-1)}P(2t,z,x)
    \ees
    for all $x$, $y\in V$ and $z\in B(y,\sqrt{t})$, where 
    $$C_5=2^{-\frac{m}{2}(r-1)}\tilde{C}_4^{r-1}2^{nr}\exp(4rD).$$ 
Thanks to $D_{\mu}<+\infty$, it follows from \eqref{3.5.} and Proposition \ref{p2.1} (iv)  that 
    \begin{equation*}
        \begin{aligned}
            \sum_{x\in V} P^r(t,y,x)\mu(x) &\le C_5 t^{-\frac{m}{2}(r-1)}\sum_{x\in V} P(2t,z,x)\mu(x)\\
            & \le C_5 t^{-\frac{m}{2}(r-1)},
        \end{aligned}
    \end{equation*}
    which implies that 
    \begin{equation*}
        \|P(t,y,\cdot)\|_{\ell^r(V)}\le C_5^{\frac{1}{r}} t^{-\frac{m}{2}\left(1-\frac{1}{r}\right)}~\text{for all}~y\in V\text{ and }~t>0.
    \end{equation*}
Taking $C_{4}=C_{5}^{\frac{1}{r}}$, then we complete the proof.
\end{proof}

With the proof of the preceding assertion concluded, we now embark on the proof of Theorem \ref{t2.1.}, which proceeds as follows.

\begin{proof}[Proof of Theorem \ref{t2.1.}]
If $b = \infty$, then by H\"{o}lder inequality (Lemma \ref{l3.2}),
\bess
\left|\sum_{y\in V}P(t,x,y)g(y)\mu(y)\right| &\leq& \left[\sum_{y\in V}P(t,x,y)^r\mu(y)\right]^{\frac{1}{r}}\left[\sum_{y\in V}g^a(y)\mu(y)\right]^{\frac{1}{a}}\nm\\
&\leq& C_4t^{-\frac{m}{2}(1 - \frac{1}{r})}\|g\|_{\ell^a(V)}.
\eess

If $1<b<\infty$, then 
\bes\label{4.4}
r\leq b \text{ and } a\leq b.
\ees Thus
$$
U(x,y):=|P(t,x,y)|^{\frac{r}{b}}|g(y)|^{\frac{a}{b}}, \quad V(x,y):=|P(t,x,y)|^{1 - \frac{r}{b}}|g(y)|^{1 - \frac{a}{b}}
$$
is well-defined.
Define
$$m:=\begin{cases}
\frac{r}{(1 - \frac{r}{b})b'}~,&\text{ if }r<b,\\
\infty,&\text{ if }r=b,
\end{cases}
 \quad 
 n:=\begin{cases}
  \frac{a}{(1 - \frac{a}{b})b'} ~,&\text{ if }a<b,\\
     \infty,&\text{ if }a=b,
 \end{cases}$$
where $b$ satisfies
$$\frac{1}{b}+\frac{1}{b^{'}}=1.$$
Obviously, 
$$m,n\in[1,+\infty], ~\frac{1}{m}+\frac{1}{n}=1.$$
Then, when $m,n<\infty$, by H\"{o}lder inequality, we deduce that

\bess
    \|V(x,\cdot)\|_{\ell^{b'}(V)}&=&\left[\sum_{y\in V}|P(t,x,y)|^{(1 - \frac{r}{b})b'}|g(y)|^{(1 - \frac{a}{b})b'}\mu(y)\right]^{\frac{1}{b'}}\nm\\
    &\leq&\left[\sum_{y\in V}|P(t,x,y)|^r\mu(y)\right]^{\frac{1}{mb'}}\left[\sum_{y\in V}|g(y)|^a\mu(y)\right]^{\frac{1}{nb'}}\nm\\
    &=&\|P(t,x,\cdot)\|_{\ell^r(V)}^{\frac{r}{mb'}}\|g\|_{\ell^a(V)}^{\frac{a}{nb'}}\nm\\
    &=&\|P(t,x,\cdot)\|_{\ell^r(V)}^{1 - \frac{r}{b}}\|g\|_{\ell^a(V)}^{1 - \frac{a}{b}}. 
\eess
When $m=\infty$ or $n=\infty$, direct calculation yields that
\bes
\|V(x,\cdot)\|_{\ell^{b'}(V)}\le \|P(t,x,\cdot)\|_{\ell^r(V)}^{1 - \frac{r}{b}}\|g\|_{\ell^a(V)}^{1 - \frac{a}{b}}\label{4.5}
\ees
still holds.
Due to
\bess
U(x,y)V(x,y)=|P(t,x,y)||g(y)|,
\eess
by H\"{o}lder inequality, we know that
\bes
\sum_{y\in V}|P(t,x,y)g(y)|\mu(y)\leq\|V(x,\cdot)\|_{\ell^{b'}(V)}\|U(x,\cdot)\|_{\ell^{b}(V)}.\label{4.6}
\ees
From  \eqref{4.5} and \eqref{4.6}, we know that
\bess
\left[\sum_{y\in V}\left|P(t,x,y)g(y)\right|\mu(y)\right]^b\leq\|P(t,x,\cdot)\|_{\ell^r(V)}^{b - r}\|g\|_{\ell^a(V)}^{b - a}\|U(x,\cdot)\|_{\ell^b(V)}^b.
\eess
By Theorem \ref{l3.6}, it follows that
\bes
    \sum_{x\in V}\left[\sum_{y\in V}|P(t,x,y)g(y)|\mu(y)\right]^b\mu(x)&\leq&\|P(t,x,\cdot)\|_{\ell^r(V)}^{b - r}\|g\|_{\ell^a(V)}^{b - a}\sum\limits_{x\in V}\sum\limits_{y\in V}U^b(x,y)\mu(y)\mu(x)\nm\\
    &\leq& C_4^{b - r}t^{-\frac{m}{2}(1 - \frac{1}{r})(b-r)}\|g\|_{\ell^a(V)}^{b-a}\sum\limits_{x\in V}\sum_{y\in V}U^b(x,y)\mu(y)\mu(x).\label{4.6,}
\ees
Due to $g\in\ell^a(V)$, we conclude that
\bess
    \sum_{x\in V}\sum_{y\in V}U^b(x,y)\mu(y)\mu(x)&=&\sum_{x\in V}\sum_{y\in V}P^r(t,x,y)|g(y)|^a\mu(y)\mu(x)\nm\\
    &=&\sum_{y\in V}\sum_{x\in V}P^r(t,x,y)|g(y)|^a\mu(x)\mu(y)\nm\\
    &\leq& C_4^{r}t^{-\frac{m}{2}(1 - \frac{1}{r})r}\sum_{y\in V}|g(y)|^a\mu(y).
\eess
Combining this with \eqref{4.6,}, we get
\bess
\sum_{x\in V}\left[\sum_{y\in V}|P(t,x,y)g(y)|\mu(y)\right]^b\mu(x)\leq C_4^{b}t^{-\frac{m}{2}(1- \frac{1}{r})b}\|g\|_{\ell^a(V)}^b.
\eess
This implies  \eqref{2.9} holds.

If $b=1$, then by \eqref{4.4}, we have $a=1=r$.
Then by Lemma \ref{l3.9}, we deduce that

\bess
\left\|\sum_{y\in V}P(t,\cdot,y)g(y)\mu(y)\right\|_{\ell^1(V)}
\leq \|g\|_{\ell^1(V)}.
\eess

Therefore, we get the desired result \eqref{2.9}.

    We now complete the proof.
\end{proof}

\begin{remark}
 We also could use Riesz-Thorin interpolation theorem $(\text{see \cite{BL}})$  to prove Theorem \ref{t2.1.}.
\end{remark}

\subsection{Proof of Theorem \ref{t2.3}}
\begin{proof}[Proof of Theorem \ref{t2.3}]
    The proof of Theorem \ref{t2.3} follows almost entirely analogously to those of Theorems \ref{l3.6} and \ref{t2.1.}, with the sole distinction that we employ Propositions \ref{p2.4} and \ref{p2.7} instead of Propositions \ref{p2.3} and \ref{p2.6}. Naturally, this entails replacing all instances of Propositions \ref{p2.3} and \ref{p2.6} in the proofs of Theorems \ref{l3.6} and \ref{t2.1.} with Propositions \ref{p2.4} and \ref{p2.7}, respectively.
\end{proof}

\section{Proof of Theorem \ref{t2.2}}

Consider the following problem
\begin{equation}\label{3.26.}
    \left\{\begin{array}{l}
        u(t,x)=\sum\limits_{y\in V}P(t,x,y)u_0(y)\mu(y)+\int_0^t\sum\limits_{y\in V}P(t-s,x,y)v^p(s,y)\mu(y)ds,~0<t\le T, x\in V, \\
        v(t, x)=\sum\limits_{y \in V} P(t, x, y) v_0(y) \mu(y)+\int_0^t \sum\limits_{y \in V} P(t-s, x, y) u^q(s, y) \mu(y) d s,~0<t\le T, x\in V, \\
        u(0, x)=u_0(x), \quad x \in V, \\
        v(0, x)=v_0(x), \quad x \in V .
    \end{array}\right.
\end{equation}
In order to prove Theorem \ref{t2.2}, we neee the following results, which implies a solution to \eqref{1.4.} satisfies \eqref{3.26.}.
\begin{lem}\label{l3.12}
    Suppose $D_{\mu}<+\infty$ and $(u, v)$ is a solution to the problem \eqref{1.4.}. Then $(u,v)$ satisfies the problem \eqref{3.26.}. 
\end{lem}
\begin{proof}
    It follows from \eqref{1.4.} that
    \bes
    \int_0^t \sum_{y\in V} P(t-s,x,y)u_s(s,y)\mu(y)ds&=&\int_0^t \sum_{y\in V} P(t-s,x,y)\Delta u(s,y)\mu(y)ds\nm\\[2mm]
    &&+\int_0^t \sum_{y\in V}P(t-s,x,y)v^p(s,y)\mu(y)ds\label{3.40}.
    \ees
    By integrating by parts(see page 10 in \cite{LWcv}), we get
    \bes
    \int_0^t \sum_{y\in V}P(t-s,x,y)\Delta u(s,y)\mu(y)ds&=&\int_0^t \sum_{y\in V}\Delta_{y}P(t-s,x,y)u(s,y) \mu(y)ds\nm\\[2mm]
    &=&\int_0^t \sum_{y \in V} P_t(t-s,x,y)u(s,y)\mu(y)ds.\label{3.41}
    \ees
    By Proposition \ref{p2.1}, we deduce that 
    \bes
    \sum_{y\in V}\int_0^t\frac{\partial}{\partial s}[P(t-s,x,y)u(s,y)]ds\mu(y)&=&\sum_{y\in V}[P(0,x,y)u(t,y)-P(t,x,y) u(0,y)]\mu(y)\nm\\[2mm]
    &=&u(t,x)-\sum_{y \in V} P(t,x,y)u_0(y)\mu(y).\label{3.42}
    \ees
    Combining \eqref{3.40}, \eqref{3.41} with \eqref{3.42}, we see that $u(t, x)$ satisfies the first equation in \eqref{3.26.}. Similarly, we could see that $v(t,x)$ satisfies the second equation in \eqref{3.26.}.
\end{proof}

\begin{proof}[Proof of Theorem \ref{t2.2}]
We claim that there exists $(u,v)\in C([0,T];\ell^{\infty}(V))\times C([0,T];\ell^{\infty}(V))$ so that $(u,v)$ satisfies \eqref{3.26.} for some $T>0$.

{\bf Step 1} . Suppose first that $p, q \ge 1$. We show that there admits a unique nonnegative solution of \eqref{3.26.}. Let $M>\max \left\{\left\|u_0\right\|_{\ell^{\infty}(V)},\left\|v_0\right\|_{\ell^{\infty}(V)}\right\}$. Then clearly, there exists a sufficiently small
$T=T(M, p, q)>0$ \text { such that }
\begin{equation}\label{3.28,}
\begin{aligned}
 \left\|u_0\right\|_{\ell^{\infty}(V)}+t M^p<M, \text{ and }
 \left\|v_0\right\|_{\ell^{\infty}(V)}+t M^q<M \text { for all } t \in[0, T] \text {. }
\end{aligned}\end{equation}
 Set $E_T=\left\{(u, v) \mid u, v \in C\left([0,T];\ell^{\infty}(V)\right),\|(u, v)\|<+\infty\right\}$ and
    $$A_{T}=\{(u,v)\in E_T: 0\le u\le M, 0\le v\le M\},$$ where
$$
\|(u, v)\|=\max_{t \in[0,T]}\left\{\|u(t,\cdot)\|_{\ell^{\infty}(V)}+\|v(t,\cdot)\|_{\ell^{\infty}(V)}\right\}.
$$
    It is easy to check that $E_{T}$ is a Banach space and $A_{T}$ is a closed subset of $E_{T}$.
For any $(u,v)\in A_{T}$, 
    define an operator $\psi(u, v):=\left(\Phi_1(v), \Phi_2(u)\right)$,
    where
   \bes
        &&\Phi_1(v)(t,x)=\sum_{y \in V} P(t, x, y) u_0(y) \mu(y)+\int_0^t \sum_{y \in V} P(t-s, x, y) v^p(s,y)\mu(y) ds, \label{3.18}\\
        &&\Phi_2(u)(t,x)=\sum_{y \in V} P(t, x, y) v_0(y) \mu(y)+\int_0^t \sum_{y \in V} P(t-s, x, y) u^q(s,y)\mu(y) ds.\label{3.19}
\ees
For $t \in[0, T), x \in V$ and any $\delta \in(0, T-t)$, by Proposition \ref{p2.2},
$$
\begin{aligned}
I(t, x)=I_{\delta}(t, x): & =\sum_{y \in V} P(t+\delta, x, y) u_0(y) \mu(y)-\sum_{y \in V} P(t, x, y) u_0(y) \mu(y) \\
    & =\sum_{k=0}^{+\infty} \frac{(t+\delta)^k\left(\Delta^k u_0\right)(x)}{k !}-\sum_{k=0}^{+\infty} \frac{t^k\left(\Delta^k u_0\right)(x)}{k !} \\
    & =\sum_{k=0}^{+\infty} \frac{(t+\delta)^k-t^k}{k !} \Delta^k u_0(x) .
\end{aligned}
$$
Recalling that $u_0$ is bounded, we may assume that 
\bes 
\left|u_0(x)\right| \leq B_1 \text{ on V }. \label{3.20}
\ees
Then
$$
\begin{aligned}
    |\Delta u_0(x)| & =\left|\frac{1}{\mu(x)} \sum_{y \sim x} \omega_{x y}\left(u_0(y)-u_0(x)\right)\right| \\
    & \leq 2 D_\mu B_1 .
\end{aligned}
$$
By induction, we have for any $k \in \mathbb{Z}^{+}$and $x \in V$,
$$
\left|\Delta^k u_0(x)\right| \le 2^k D_\mu^k B_1 .
$$
Thus for any $t\in[0,T),~x\in V$,
\bes\label{3.21}
|I_{\delta}(t, x)| &\leq& \sum_{k=0}^{+\infty} \frac{(t+\delta)^k-t^k}{k!}2^k D_\mu^kB_{1}\nm\\
    &=&B_{1}\left(e^{2D_\mu(t+\delta)}-e^{2D_\mu t}\right).
\ees
Hence, there exists a sufficiently small $\delta=\delta(\epsilon, T)>0$ so that
\bes
|I_{\delta}(t, x)|<\epsilon \label{3.31}
\ees
 for $t \in[0,T)$ and $x \in V$. By similar arguments as above, we can show that there exists $\delta=\delta(\epsilon)\ll1$ so that
\bes
\left|\sum\limits_{y\in V}P(T,x,y)u_0(y)\mu(y)-\sum_{y\in V}P(T-\delta,x,y)u_0(y)\mu(y)\right|<\epsilon.\label{3.32.}
\ees
Combining \eqref{3.31} with \eqref{3.32.}, one can obtain $$\sum\limits_{y\in V}P(t,x,y)u_{0}(y)\mu(y)\in C\left([0,T];\ell^{\infty}(V)\right).
$$
It follows that 
$$\sum\limits_{y\in V}P(t,\cdot,y)u_{0}(y)\mu(y) \text{ is strongly measurable w.r.t.} t\in[0,T]\text{ in }\ell^{\infty}(V).$$

Let $t\in[0,T),x\in V,~\delta\in(0,T-t)$. Define 
$$
J(t, x)=J(t,x;v)=\int_0^t \sum\limits_{y \in V} P(t-s,x,y) v^p(s, y) \mu(y) ds.
$$
Then
\bes
&&J(t+\delta,x)-J(t,x)\nm\\
&=&\int_0^{t+\delta}\sum_{y\in V}P(t+\delta-s,x,y)v^p(s,y)\mu(y)ds-\int_0^t \sum_{y \in V} P(t-s,x,y)v^p(s,y)\mu(y)ds\nm\\
&=&\int_0^{t+\delta} \sum_{y \in V} P(t+\delta-s, x, y) v^p(s, y) \mu(y) ds-\int_0^{t} \sum_{y \in V} P(t+\delta-s, x, y) v^p(s, y) \mu(y) d s\nm\\
&&+\int_0^{t} \sum_{y \in V} P(t+\delta-s, x, y) v^p(s, y) \mu(y) d s-\int_0^t \sum_{y \in V} P(t-s, x, y) v^p(s, y) \mu(y) d s\nm\\
&=&\int_0^{t} \sum_{y \in V} P(t+\delta-s, x, y) v^p(s, y) \mu(y)-\sum_{y \in V} P(t-s, x, y) v^p(s, y) \mu(y) ds\nm\\
&&+\int_t^{t+\delta} \sum_{y \in V} P(t+\delta-s, x, y) v^p(s, y) \mu(y) ds\nm.
\ees
Define
\bess I:=\int_0^{t} \sum_{y \in V} P(t+\delta-s, x, y) v^p(s, y) \mu(y)-\sum_{y \in V} P(t-s, x, y) v^p(s, y) \mu(y) ds\nm\\
\eess
and
\bess
II:=\int_t^{t+\delta} \sum_{y \in V} P(t+\delta-s, x, y) v^p(s, y) \mu(y) ds.
\eess
 Then by \eqref{2.3} and the mean value theorem, we have
\bes
I &=&\int_0^{t} \sum_{y \in V} P(t+\delta-s, x, y) v^p(s, y) \mu(y) d s-\sum_{y \in V} P(t-s, x, y) v^p(s, y) \mu(y)ds\nm\\
&=&\int_0^{t} \sum_{k=0}^{+\infty} \frac{(t+\delta-s)^k \Delta^k v^p(s, y)}{k !}-\sum_{k=0}^{+\infty} \frac{(t-s)^k \Delta^k v^p(s, y)}{k !} ds\nm\\
&=&\int_0^{t} \sum_{k=0}^{+\infty} \frac{(t+\delta-s)^k-(t-s)^k}{k !} \Delta^k v^p(s, y) ds\nm\\
&=& \int_{0}^{t} \sum_{k=1}^{\infty} \frac{k\theta^{k - 1}\delta}{k!} \Delta^{k} v^{p}(s, y)ds,\label{5.12}
\ees
where $\theta \in (t - s, t - s + \delta)$.
Let $$B_{2}:=\|v\|_{C\left([0,T];\ell^{\infty}(V)\right)}^p.$$ 
Direct calculation yields that
\bes
 \sum_{k=1}^{\infty} \frac{\theta^{k - 1}\delta}{(k - 1)!} \Delta^{k} v^{p}(s, y) 
&\leq&  \sum_{k=1}^{\infty} \frac{|\theta|^{k - 1}\delta}{(k - 1)!} 2^{k} D_{\mu}^{k} B_{2} \nm \\
&=&  \sum_{k=0}^{\infty} \frac{|\theta|^{k}\delta}{k!} 2^{k + 1} D_{\mu}^{k + 1} B_{2} \nm \\
&=&  e^{|\theta| 2 D_{\mu}} 2 D_{\mu} \delta B_{2}.  \label{5.13.}
\ees
Combining \eqref{5.12} with \eqref{5.13.}, we conclude that
$$I\le\int_{0}^{t} e^{|\theta| 2 D_{\mu}} 2 D_{\mu} \delta B_{2} ds\le e^{(t + \delta)2 D_{\mu}} 2 D_{\mu} \delta B_{2}t. $$
By Proposition \ref{p2.1} (iv), we have
\bess
|II| &=&\left|\int_t^{t+\delta} \sum_{y \in V} P(t+\delta-s, x, y) v^p(s, y) \mu(y) d s\right| \\
&\leq& \int_t^{t+\delta} \sum_{y \in V} P(t+\delta-s, x, y) \mu(y) B_2 d s \\
 &=&B_2 \delta.
    \eess
Hence, we deduce that there exists $\delta_{1}=\delta_{1}(\epsilon,T)>0$ such that
\begin{equation}\label{3.37,}
\|J(t+\delta,\cdot)-J(t,\cdot)\|_{\ell^{\infty}(V)}<\epsilon
\end{equation}
for all $\delta \in\left(0, \delta_1\right), t \in[0, T)$.
Similarly, we can find $\delta_2=\delta_2(\epsilon, T) \in\left(0, \delta_1\right)$ so that
\begin{equation}\label{3.38,}
\| J(T, \cdot)-J\left(T-\delta, \cdot\right)\|_{\ell^{\infty}(V)}<\epsilon
\end{equation}
for all $\delta \in\left(0, \delta_2\right)$. Therefore, we get $J$ is continuous from $[0,T]$ to $\ell^{\infty}(V)$. It follows that \begin{equation}\label{5.13}
    J(t,x) \text{~ is strongly measurable w.r.t.~ } t\in[0,T]\text{ in }\ell^{\infty}(V).
\end{equation}
Similarly, we may show that 
\begin{equation}\label{5.14}
    \int_0^t \sum\limits_{y \in V} P(t-s,x,y) u^q(s, y) \mu(y) ds \text{~ is strongly measurable w.r.t.~ } t\in[0,T]\text{ in  }\ell^{\infty}(V).
\end{equation}
is strongly measurable w.r.t. $t\in[0,T]$ in $\ell^{\infty}(V)$.
By \eqref{3.18}, \eqref{3.19}, Proposition \ref{p2.1} (iv) and Lemma \ref{t2.1}, we deduce that for any $t\in[0,T]$,
\bes\label{3.26}
\left\|\Phi_1(v)(t, \cdot)\right\|_{\ell^{\infty}(V)} &\le&\left\|u_0\right\|_{\ell^{\infty}(V)}+\int_0^t \sum_{y \in V} P(t-s, x, y)\|v(s, \cdot)\|_{\ell^{\infty}(V)}^p \mu(y) ds\nm \\
&\le&\left\|u_0\right\|_{\ell^{\infty}(V)}+\int_0^t \| v\left(s, \cdot\right) \|_{\ell^{\infty}(V)}^p ds, \nm\\
&\le&\left\|u_0\right\|_{\ell^{\infty}(V)}+\left(\max_{s \in[0, T]} \| v\left(s,\cdot\right) \|_{\ell^{\infty}(V)}\right)^pT\nm\\
&\le&\left\|u_0\right\|_{\ell^{\infty}(V)}+M^{p}T<\infty
\ees
and
\bes\label{3.27}
\left\|\Phi_2(u)(t, \cdot)\right\|_{\ell^{\infty}(V)} &\le&\left\|v_0\right\|_{\ell^{\infty}(V)}+\int_0^t \sum_{y \in V} P(t-s, x, y) \| u\left(s,\cdot\right)\|_{\ell^{\infty}(V)}^q \mu(y) ds\nm \\
&\le&\left\|v_0\right\|_{\ell^{\infty}(V)}+\int_0^t\|u(s, \cdot)\|_{\ell^{\infty}(V)}^q ds\nm \\
&\le&\left\|v_0\right\|_{\ell^{\infty}(V)}+\left(\max_{s \in[0, T]} \| u\left(s,\cdot\right) \|_{\ell^{\infty}(V)}\right)^qT\nm\\
&\le&\left\|v_0\right\|_{\ell^{\infty}(V)}+M^{q}T<\infty.\label{}
\ees
 Thus, we have
\bes\Phi_1(v)\in C\left([0, T]; \ell^{\infty}(V)\right).\label{3.28}
\ees 
Similarly, we may show that $\Phi_2(u) \in C\left([0, T]; \ell^{\infty}(V)\right)$
by virtue of \eqref{3.27}.    Furthermore, since
    $$
    \begin{aligned}
        \|\psi(u, v)\| & =\left\|\left(\Phi_1(v), \Phi_2(u)\right)\right\| \\
        & =\max_{t \in[0, T]}\left\{\left\|\Phi_1(v)(t, \cdot)\right\|_{\ell^{\infty}(V)}+\left\|\Phi_2(u)(t, \cdot)\right\|_{\ell^{\infty}(V)}\right\},
    \end{aligned}
    $$
    by \eqref{3.26} and \eqref{3.27}, one may get
 \bess
\|\psi(u,v)\|&\le& \left\|u_0\right\|_{\ell^{\infty}(V)}+\left\|v_0\right\|_{\ell^{\infty}(V)}+\int_0^t\|v(s, \cdot)\|_{\ell^{\infty}(V)}^p+\left\|u(s,\cdot) \right\|_{\ell^{\infty}(V)}^q d s \nm\\
&\leq&\|(u, v)\|+T\|(u, v)\|^{p}+T\|(u, v)\|^{q}<+\infty.
\eess
On the other hand, by \eqref{3.26}, \eqref{3.27} and \eqref{3.28,}, we can obtain 
\bes\label{3.39.}
\left|\Phi_1(v)(t,x)\right| \le M,~\left|\Phi_2(u)(t,x)\right| \le M
\ees
for $t \in[0, T]$, $x \in V$. Recalling that $u_0, v_0 \ge 0$ and $u, v \in A_T$, by the first two equations of \eqref{3.26.}, it can be concluded that $\Phi_1(v)(t, x), \Phi_2(u)(t, x) \ge 0$ for $t \in[0, T]$ and $x \in V$. Thus, combining this with \eqref{3.39.} yields
\begin{equation}\label{3.40.}
0 \leq \Phi_1(v)(t, x), \Phi_2(u)(t, x) \leq M \text { for } t \in[0, T] \text { and } x \in V \text {. }
\end{equation}
Therefore, we know that $\psi:A_{T}\to A_{T}$. 

For any $\left(u_1, v_1\right),\left(u_2, v_2\right) \in E_T$, by \eqref{3.18}, \eqref{3.19}, Lemma \ref{t2.1} and Proposition \ref{p2.1} (iv), one can deduce that
\bes
&&\quad\left\|\psi\left(u_1, v_1\right)-\psi\left(u_2, v_2\right)\right\|\nm\\[1mm]
&&=\left\|\left(\Phi_1\left(v_1\right), \Phi_2\left(u_1\right)\right)-\left(\Phi_1\left(v_2\right), \Phi_2\left(u_2\right)\right)\right\|\nm\\[1mm]
&&=\left\|\left(\Phi_1\left(v_1\right)-\Phi_1\left(v_2\right), \Phi_2\left(u_1\right)-\Phi_2\left(u_2\right)\right)\right\|\nm\\[1mm]
&&=\max _{t \in[0, T]}\left[\left\|\left(\Phi_1\left(v_1\right)-\Phi_1\left(v_2\right)\right)(t, \cdot)\right\|_{\ell^{\infty}(V)}+\left\|\left(\Phi_2\left(u_1\right)-\Phi_2\left(u_2\right)\right)(t, \cdot)\right\|_{\ell^{\infty}(V)}\right]\nm\\[1mm]
&&=\max _{t \in[0, T]}\left\{\left\|\int_0^t \sum_{y \in V} P(t-s, \cdot, y)\left[v_1^p(s, y)-v_2^p(s, y)\right] \mu(y) d s\right\|_{\ell^{\infty}(V)}\right.\nm\\
&&\left.+\left\|\int_0^t \sum_{y \in V} P(t-s, \cdot, y)\left[u_1^q(s, y)-u_2^q(s, y)\right] \mu(y) d s\right\|_{\ell^{\infty}(V)}\right\}\nm\\[1mm]
&&\leq \max _{t \in[0, T]}\left\{\int_0^t\left\|\sum_{y \in V} P(t-s,\cdot, y)\left[v_1^p(s, y)-v_2^p(s, y)\right] \mu(y)\right\|_{\ell^{\infty}(V)} d s\right.\nm\\[1mm]
&&\left.+\int_0^t\left\|\sum_{y \in V} P(t-s,\cdot, y)\left[u_1^q(s,y)-u_2^q(s, y)\right] \mu(y)\right\|_{\ell^{\infty}(V)} ds\right\}\nm\\[1mm]
&&\le\max_{t \in[0, T]}\left\{\int_0^t\left\|v_1^p\left(s,\cdot\right)-v_2^p(s, \cdot)\right\|_{\ell^{\infty}(V)}d s+\int_0^t\left\|u_1^q\left(s,\cdot\right)-u_2^q(s,\cdot)\right\|_{\ell^{\infty}(V)}ds\right\}.\label{3.22}
\ees
When $p> 1$, by the mean value theorem, we may conclude that
\bes
\left\|v_1^p(s, \cdot)-v_2^p(s, \cdot)\right\|_{\ell^{\infty}(V)} &&=\left\|p \theta^{p-1}(s, \cdot)\left(v_1-v_2\right)(s, \cdot)\right\|_{\ell^{\infty}(V)} \nm\\
        &&\leq\left\|p\left|v_1(s, \cdot)+v_2(s, \cdot)\right|^{p-1}\left(v_1-v_2\right)(s, \cdot)\right\|_{\ell^{\infty}(V)}\nm \\
&&\le p\left\|v_1\left(s,\cdot\right)+v_2(s, \cdot)\right\|_{\ell^{\infty}(V)}^{p-1}\left\|\left(v_1-v_2\right)\left(s, \cdot\right)\right\|_{\ell^{\infty}(V)},\label{3.23}
    \ees
where $\theta$ is between $v_1$ and $v_2$. Thus, we deduce that when $p\ge 1$, 
\bes
\left\|v_1^p(s, \cdot)-v_2^p(s, \cdot)\right\|_{\ell^{\infty}(V)} 
\le p\left\|v_1\left(s,\cdot\right)+v_2(s, \cdot)\right\|_{\ell^{\infty}(V)}^{p-1}\left\|\left(v_1-v_2\right)\left(s, \cdot\right)\right\|_{\ell^{\infty}(V)},\label{5.25}
\ees
    Similarly, we have
\bes
        \left\|u_1^q(s,\cdot)-u_2^q(s,\cdot)\right\|_{\ell^{\infty}(V)} \le q\left\|u_1(s, \cdot)+u_2(s, \cdot)\right\|_{\ell^{\infty}(V)}^{q-1}\left\|\left(u_1-u_2\right)(s, \cdot)\right\|_{\ell^{\infty}(V)}.\label{3.24} 
\ees
By \eqref{3.22}, \eqref{3.23} and \eqref{3.24}, we know that
\bes
&&\quad \left\|\left(\psi\left(u_1, v_1\right)-\psi\left(u_2, v_2\right)\right)\right\|\nm\\
&&\le\max _{t \in[0, T]}\left\{\int_0^t p\left\|v_1\left(s,\cdot\right)+v_2\left(s,\cdot\right)\right\|_{\ell^{\infty}(V)}^{p-1}\left\|\left(v_1-v_2\right)\left(s,\cdot\right)\right\|_{\ell^{\infty}(V)}ds\right. \nm\\
&&\qquad\qquad\left.+\int_0^t q\left\|u_1(s,\cdot)+u_2(s, \cdot)\right\|_{\ell^{\infty}(V)}^{q-1}\left\|\left(u_1-u_2\right)(s, \cdot)\right\|_{\ell^{\infty}(V)} ds\right\}\nm \\
&&\le\max_{t \in[0, T]}\left\{\int_0^tB_{p,q,M}
\left(\left\|\left(v_1-v_2\right)(s, \cdot)\right\|_{\ell^{\infty}(V)}+\left\|\left(u_1-u_2\right)(s,\cdot)\right\|_{\ell^{\infty}(V)}\right) d s\right\}\nm \\
&&\le TB_{p,q,M}\|\left(u_1, v_1\right)-\left(u_2, v_2\right) \|,\nm\label{3.25}
\ees
where 
$$B_{p,q,M}=\max\left\lbrace p(2M)^{p-1},q(2M)^{q-1}\right\rbrace.$$
    Clearly, there exists a sufficiently small $T_0>0$ so that $T\max\left\lbrace p(2M)^{p-1},q(2M)^{q-1}\right\rbrace<1$ for all $T<T_0$. By \eqref{3.25}, 
this implies that $\psi$ is a strict contraction map from $A_T$ to $A_T$ for all $T<T_0$.
    
Next we suppose that $T<T_0$. By Banach's Fixed Point Theorem, $\psi$ has a unique fixed point $(u,v)\in A_T$.  Hence $u,v\in C\left([0, T] ; \ell^{\infty}(V)\right)$ and $(u,v)$ satisfies \eqref{3.26.}.
It follows that $u(t,x), v(t,x)$ are continuous w.r.t. $t$ for all fixed $x\in V$. Taking the derivative of $t$ on both sides of the first two equations of \eqref{3.26.}, by similar arguments as in the proof of Theorem 3.2 in \cite{HuWang}, we discover $(u,v)$ is a solution of \eqref{1.4.}. Hence, by Lemma \ref{l3.12}, we know that $(u,v)$ is the unique nonnegative solution to \eqref{1.4.}.

{\bf Step 2}. We now suppose that $0<p<1, q \ge 1$. Clearly, we could choose a sequence of functions $\left\{g_n(x)\right\}_{n=1}^{+\infty}$ so that for each fixed $n\in\mathbb{Z}^{+}$, $g_{n}(x)$ is nondecreasing and globally Lipschitz continuous w.r.t. $x$ and 
\begin{equation*}
g_n(x)=
\begin{cases}
    0, & x\leq0 ,\\
    x^{p}, & x \ge \frac{1}{2n}.
\end{cases}
\end{equation*}
By similar arguments as in step 1, there exist $T>0$ and a nonnegative solution $
\left(u_n(t,x),v_n(t,x)\right)$ to
\begin{equation}\label{3.33}
\left\{\begin{array}{l}
u_{n}(t, x)=\sum\limits_{y \in V} P(t, x, y) u_0(y) \mu(y)+\int_0^t \sum\limits_{y \in V} P(t-s, x, y) g_{n}(v_{n}(s, y)) \mu(y) d s,~0<t<T, x\in V, \\
v_{n}(t, x)=\sum\limits_{y \in V} P(t, x, y)(v_0(y)+\frac{1}{n}) \mu(y)+\int_0^t \sum\limits_{y \in V} P(t-s, x, y) u_{n}^q(s, y) \mu(y) d s,~0<t<T, x\in V, \\
u_{n}(0, x)=u_0(x), \quad x \in V, \\
v_{n}(0, x)=v_0(x)+\frac{1}{n}, \quad x \in V
\end{array}\right.
\end{equation} 
satisfying
\begin{equation}\label{3.32}
	u_n(t,x),v_n(t,x)\in C\left([0,T];\ell^{\infty}(V)\right).
\end{equation}
and 
\begin{equation}\label{3.49,}
 0\le u_n(t,x),v_n(t,x)\le M \text { for } t \in [0,T] \text { and } x \in V.
\end{equation}
From \eqref{3.32} and Lemma \ref{l3.10}, we see that if $m \leq n$, then
\bes
u_m(t,x) \ge u_n(t,x), v_m(t,x) \ge v_n(t,x) \text { for } t \in [0,T] \text { and } x \in V.\label{3.34}
\ees
Thus, one may define
\bes
    u(t,x)=\lim _{n \rightarrow+\infty} u_n(t,x), v(t,x)=\lim _{n \rightarrow+\infty} v_n(t,x).\label{3.35}
\ees
Then by a similar discussion as in the proof of \eqref{3.37,} and \eqref{3.38,}, we may show that
for any \(t\in[0,T]\), there exists \(\delta = \delta(\epsilon)\) so that
for all \(\tau\in N_t(\delta)\),
\[
\left\lVert u_n(t,\cdot)-u_n(\tau,\cdot)\right\rVert_{\ell^\infty(V)}<\epsilon,\quad\forall n\geq1,
\]
where \(N_t(\delta)\) is a neighborhood of \(t\) with \(\delta\) as the radius. This implies that
 $\{u_n\}$ is an equicontinuous sequence in $C([0,T];\ell^{\infty}(V))$. Thus, by Theorem 5.1.4 in \cite{XMY}, we see that \begin{equation}
    u\in C([0,T];\ell^{\infty}(V)).
\end{equation}
 Similarly, we may show that
 \begin{equation}
    v\in C([0,T];\ell^{\infty}(V)).
\end{equation}
Thanks to $v_0(x)\ge 0$ for $x\in V$ and $u_n\ge 0$, by the second equality in \eqref{3.33}, we deduce that $v_n(t, x) \ge \frac{1}{2 n}$ for $t \in[0, T]$ and $x \in V$.
    Hence \begin{equation}\label{3.36}
        \left\{\begin{array}{l}
            u_n(t, x)=\sum\limits_{y \in V} P(t, x, y) u_0(y) \mu(y)+\int_0^t \sum\limits_{y \in V} P(t-s, x, y)v_{n}^{p}(s, y) \mu(y) d s,~0<t<T, x\in V, \\
            v_n(t, x)=\sum\limits_{y \in V} P(t, x, y)(v_0(y)+\frac{1}{n}) \mu(y)+\int_0^t \sum\limits_{y \in V} P(t-s, x, y) u_{n}^q(s, y) \mu(y) d s,~0<t<T, x\in V, \\
            u_n(0, x)=u_0(x), \quad x \in V, \\
            v_n(0, x)=v_0(x)+\frac{1}{n}, \quad x \in V
        \end{array}\right.
    \end{equation}
 for $n \ge 1$.
By \eqref{3.49,} and Proposition \ref{p2.1} (iv), we know that
\bes\sum\limits_{y\in V}P(t-s,x,y)v_n^p(s,y)\mu(y)\le M^p~\text{and}~ 
    \sum_{y\in V}P(t-s,x,y)u_n^q(s,y)\mu(y)\le M^q.\label{3.39}
\ees
Thus, letting $n \rightarrow+\infty$ in \eqref{3.36}, in view of \eqref{3.39}, by Lebesgue Dominated Convergence Theorem, we obtain $(u,v)$ satisfies \eqref{3.26.}.
 
By \eqref{3.35} and \eqref{3.49,}, one may get
    $$
0\le u(t,x)\le M~\text{and}~0\le v(t,x)\le M.
    $$
    In view of these, by similar arguments as in the proof of Theorem 3.2 in \cite{HuWang}, we see that $(u(t,x),v(t,x))$ is a solution to \eqref{1.4.}.
\end{proof}

\section{Proof of Theorem \ref{t2.5.}}
Without loss of generality, we assume $p\le q$ throughout this section.

To prove Theorem \ref{t2.5.}, we need the following result.
\begin{lem}\label{l6.1}
Let $G$ be a graph satisfying CDE$(n,0)$ and the condition {\bf (LVG$m$)}.
Assume
$$ \mu_{\min}>0,~ \mu_{\max}<\infty,~ \omega_{\min}>0~ \text{ and } D_{\mu}<\infty.$$ Let $p,q\ge 1, pq>1, p\le q, r_1\ge 1$, $r_2 \ge 1$, $T>0$. Suppose $(\tilde{u}(t,x),\tilde{v}(t,x))$ is a solution to \eqref{1.4.} for $t\in[0, T]$ and $x \in V$ satisfying
    \bes
    u_0 \in \ell^{r_1}(V), \quad v_0 \in \ell^{r_2}(V).\label{6.1}
    \ees
Assume $s_1\ge\max\{r_2,r_1\}$. Then there exists $T_1\le T$ such that
$$
\tilde{u}\in C\left(\left[0, T_1\right];\ell^{s_1}(V)\right)\text { and }\tilde{v}\in C\left(\left[0,T_1\right]; \ell^{s_1}(V)\right).
    $$
\end{lem}

\begin{proof}
{\bf Step 1}. Let $M>\max \left\{\left\|u_0\right\|_{\ell^{\infty}(V)},\left\|v_0\right\|_{\ell^{\infty}(V)}\right\}$. Then clearly, there exists a sufficiently small
$T=T(M, p, q)>0$ \text { such that }
\begin{equation}
    \begin{aligned}
        \left\|u_0\right\|_{\ell^{\infty}(V)}+t M^p<M, \text{ and }
        \left\|v_0\right\|_{\ell^{\infty}(V)}+t M^q<M \text { for all } t \in[0, T] \text {. }
\end{aligned}\end{equation}
 Set $E_T=\left\{(u,v)|~u\in C\left([0,T];\ell^{s_1}(V)\right),v\in C\left([0,T];\ell^{s_1}(V)\right),\|(u, v)\|<+\infty\right\}$ and
    $$A_{T}=\{(u,v)\in E_T: 0\le u\le M,~0\le v\le M\},$$ where
\bes
\|(u,v)\|=\max_{t\in[0,T]}\left\{\|u(t,\cdot)\|_{\ell^{s_1}(V)}+\|v(t,\cdot)\|_{\ell^{s_1}(V)}\right\}.\label{6.2.}
\ees
    It is easy to check that $E_{T}$ is a Banach space and $A_{T}$ is a closed subset of $E_{T}$.
    For any $(u,v)\in A_{T}$, 
    recalling that the operator $\psi$ is defined as $\psi(u, v):=\left(\Phi_1(v), \Phi_2(u)\right)$, where $\Phi_1(v), \Phi_2(u)$ are defined by \eqref{3.18} and \eqref{3.19} respectively. 

{\bf Step 2}. We will show that $\deg(x)<\infty$ for all $x\in V$.

Due to $\mu_{\max}<\infty$, we see that
\[
\frac{m(x)}{\mu_{\max}}\leq\frac{m(x)}{\mu(x)}\leq D_{\mu}
\]
and hence that
\[
m(x)\leq\mu_{\max}D_{\mu}.
\]
On the other hand,
\[
m(x):=\sum_{y\sim x}\omega_{xy}>\omega_{\min}\sum_{y\sim x}1 = \omega_{\min}\deg(x).
\]
Therefore, we get
\[
\deg(x)<\frac{m(x)}{\omega_{\min}}\leq\frac{\mu_{\max}D_{\mu}}{\omega_{\min}}<\infty~\text{ for all }~x\in V.
\]

{\bf Step 3}. Fix $\delta>0$. By \eqref{6.1} and Lemma \ref{l3.5}, we deduce $u_0$, $v_0\in \ell^{\infty}(V)$. Hence, by Proposition \ref{p2.2} and the mean value theorem, we know that
    \bes
    &&\sum_{y \in V} P(t+\delta, x, y) u_{0}(y) \mu(y)-\sum_{y \in V} P(t, x, y) u_0(y) \mu(y)\nm\\
    &=&\sum_{k=0}^{+\infty} \frac{(t+\delta)^k-t^k}{k !} \Delta^k u_0(x)\nm\\
    &=&\sum_{k=1}^{\infty} \frac{k \theta^{k-1} \delta}{k !} \Delta^k u_0(x)\nm\\
    &=&\delta \sum_{k=0}^{\infty} \frac{\theta^k}{k !} \Delta^k\left(\Delta u_0\right)(x)\nm\\
    &=&\delta \sum_{y \in V} P(\theta,x,y) \Delta u_0(y)\mu(y)\label{6.2}
    \ees
for  $t\ge 0$, where $\theta=\theta(t)\in (t,t+\delta)$. By Lemma \ref{l3.9}, whenever $\Delta u_0\in \ell^{s_1}(V)$, this implies that
    \bes
    &&\left\|\sum_{y\in V}P(t+\delta,x,y)u_0(y)\mu(y)-\sum_{y\in V} P(t,x,y)u_0(y)\mu(y)\right\|_{\ell^{s_1}(V)}\nm\\
    &\le&\delta\left\|\Delta u_0(\cdot)\right\|_{\ell^{s_1}(V)}.\label{6.4.}
    \ees
We next show that $\Delta u_0\in \ell^{s_1}(V)$.  By the definition of $\Delta$, we know that
    \bes
    \left\|\Delta u_0(x)\right\|_{\ell_{x}^{s_1}(V)} &\le&\left\|\sum_{y \sim x} u_0(y) \frac{\omega_{y x}}{\mu(x)}-\sum_{y \sim x} \frac{\omega_{y x}}{\mu(x)} u_0(x)\right\|_{\ell_{x}^{s_1}(V)}\nm\\
    &\le&\left\|\sum_{y \sim x} u_0(y) \frac{\omega_{y x}}{\mu(x)}\right\|_{\ell_{x}^{s_1}(V)}+\left\|\sum_{y \sim x} \frac{\omega_{y x}}{\mu(x)} u_0(x)\right\|_{\ell_{x}^{s_1}(V)}.\label{6.4}
    \ees
    For any $x \in V$, there exists $y_x$ such that $u_0\left(y_x\right)=\max\limits_{y\in V: y \sim x} u_0$. Direct calculations yield that
    \bes
    \left\|\sum_{y \sim x} u_0(y) \frac{\omega_{y x}}{\mu(x)}\right\|_{\ell^{s_1}_{x}(V)} &\le&\left\|u_0\left(y_x\right) \sum_{y \sim x} \frac{\omega_{y x}}{\mu(x)}\right\|_{\ell_x^{s_1}(V)}\nm\\
    &\le&D_\mu\left\|u_0\left(y_x\right)\right\|_{\ell_x^{s_1}(V)}\nm\\
    &\le&D_\mu\left[\left(\sup_{x \in V} \operatorname{deg}(x)+1\right)\sum_{x\in V}u_0^{s_1}(x)\right]^{\frac{1}{s_1}}\nm\\
    &\le&D_\mu\left(\sup_{x \in V} \text{deg}(x)+1\right)^{\frac{1}{s_1}}\left\|u_0\right\|_{\ell^{s_1}(V)}\nm\\
    &=&D_\mu d_m\left\|u_0\right\|_{\ell^{s_1}(V)}\label{6.5}
    \ees
    and 
    \begin{equation}\label{6.6}
        \left\|\sum_{y\sim x} \frac{\omega_{y x}}{\mu(x)} u_0(x)\right\|_{\ell_x^{s_1}(V)} \le D_{\mu}\left\|u_0\right\|_{\ell^{s_1}(V)},
    \end{equation}
    where $d_{m}:=\left(\sup\limits_{x\in V}\text{deg}(x)+1\right) ^{\frac{1}{s_1}}$. It follows from \eqref{6.4}, \eqref{6.5} and \eqref{6.6} that 
    \bes
    \left\|\Delta u_0\right\|_{\ell^{s_1}(V)}\le D_{\mu}(1+d_{m})\|u_0\|_{\ell^{s_1}(V)}.\label{6.7}
    \ees
    Define 
    $$K(t,x):=\sum_{y \in V} P(t, x, y) u_0(y) \mu(y).$$ From \eqref{6.4.}-\eqref{6.7}, we see that 
\begin{equation}\label{6.9.}
\|K(t+\delta,\cdot)-K(t,\cdot)\|_{\ell^{s_1}(V)}\le\delta D_{\mu}(1+d_{m})\left\|u_{0}\right\|_{\ell^{s_1}(V)}
\end{equation} 
for all $t\ge 0$ and $x\in V$.
By \eqref{6.9.}, we know
\bes
    K:[0,T]\to \ell^{s_1}(V) \text{~is continuous~}.\label{6.11}
\ees
Recalling that $s_1 \ge r_1 \ge 1$ and $u_0 \in \ell^{r_1}(V)$, by Lemma 3.9 and Lemma 3.5, we see that $u_0\in\ell^{s_1}(V)$ and there exists a constant $\bar{C}>0$ so that
    $$
    \max_{t\in[0,T]}\left\|K(t,\cdot)\right\|_{\ell^{s_1}(V)}=\max_{t\in[0,T]}\left\|\sum_{y \in V} P(t, x, y) u_0(y) \mu(y)\right\|_{\ell^{s_1}_{x}(V)} \leq\left\|u_0\right\|_{\ell^{s_1}(V)} \leq\bar{C}\left\|u_0\right\|_{\ell^{r_1}(V)}<+\infty.
    $$
Combining this with \eqref{6.9.}, one may obtain 
    \bes
    K\in C\left([0, T];\ell^{s_1}(V)\right).\label{6.11.}
    \ees
    
{\bf Step 4}. By Lemma \ref{l3.4}, we see that
\bes
C\left([0,T];\ell^{s_1}(V)\right)\subset C\left([0,T];\ell^{\infty}(V)\right)\label{}
\ees
and hence that for any $u,v\in A_{T}$, we have
\bes
u, v \in C\left([0,T];\ell^{\infty}(V)\right).\label{A}
\ees
For any fixed $t\in[0,T)$, $\delta\in(0,T-t]$. 
Recalling that $p\ge 1$ and $v \in C\left([0,T]; \ell^{s_1}(V)\right)$, by Lemma \ref{l3.5}, we see that there exists a constant $C>0$ so that
\bes
\left\|v^p\right\|_{C\left([0, T] ; \ell^{s_1}(V)\right)} \le\|v\|_{C\left([0,T];\ell^{ps_1}(V)\right)}^p \le C\|v\|_{C\left([0,T];\ell^{s_1}(V)\right)}^p.\label{6.13}
\ees
When $p>1$, we have
\bess
\left\|v^p(t+\delta,\cdot)-v^p(t, \cdot)\right\|_{\ell^{s_1}(V)}&=&\left\|p\theta^{p-1}(t, \cdot)[v(t+\delta, \cdot)-v(t, \cdot)]\right\|_{\ell^{s_1}(V)}\nm\\
&\le&p\left(2\|v\|_{C\left([0,T]; \ell^{\infty}(V)\right)}\right)^{p-1}\|v(t+\delta, \cdot)-v(t, \cdot)\|_{\ell^{s_1}(V)} \label{}
\eess
by the mean value theorem and \eqref{A}, where $\theta(t, \cdot)$ is between $v(t+\delta, \cdot)$ and $v(t, \cdot)$. Hence, we know that when $p\ge 1$,
\begin{equation}\label{6.16}
\left\|v^p(t+\delta,\cdot)-v^p(t, \cdot)\right\|_{\ell^{s_1}(V)}
\le C^{*}_{1}\|v(t+\delta, \cdot)-v(t, \cdot)\|_{\ell^{s_1}(V)},
\end{equation}
where  
$$C^{*}_{1}=\max\left\{1,p\left(2\|v\|_{C\left([0,T]; \ell^{\infty}(V)\right)}\right)^{p-1}\right\}.$$
 Therefore 
\bess
v^{p}(t,\cdot) \text{ is  continuous w.r.t. } t\in[0,T) \text{ in } \ell^{s_{1}}(V).
\eess
By similar arguments as above, we may show that $v^{p}(t,\cdot) \text{ is  continuous at } t=T \text{ in } \ell^{s_{1}}(V).$
Hence, we get
\bess
v^{p}(t,\cdot) \text{ is  continuous w.r.t. } t\in[0,T] \text{ in } \ell^{s_{1}}(V).
\eess
Combining this with \eqref{6.13}, we see
\bes
v^{p}\in C\left([0,T];\ell^{s_1}(V)\right).\label{6.15.}
\ees

{\bf Step 5}. For any fixed $t\in(0,T]$, $s_1,s_2\in[0,t]$,  and $x\in V$. Direct calculations yield that
    \bes
    && \sum_{y \in V} P(t-s_1, x, y) v^p(s_1, y)\mu(y)-\sum_{y \in V} P(t-s_2, x, y) v^p(s_2, y)\mu(y) \nm\\
  &\le&\left[ \sum_{y \in V} P(t-s_1,x,y) v^p(s_1, y) \mu(y)-\sum_{y \in V} P(t-s_2,x,y) v^p(s_1, y) \mu(y)\right] \nm\\
    &&+\left[\sum_{y \in V} P(t-s_2, x, y) v^p(s_1, y) \mu(y)-\sum_{y \in V} P\left(t-s_2, x, y\right) v^p(s_2, y)\mu(y)\right] \nm\\
    &=:&I(t,s_1,s_2,x)+II(t,s_1,s_2,x).\label{6.16..}
\ees
By \eqref{6.15.} and Lemma \ref{l3.9}, we deduce that 
\bes
I(t,s_1,s_2,\cdot)~\text{ and }~II(t,s_1,s_2,\cdot)\in \ell^{s_1}(V).
\ees
Hence by triangle inequality, we know 
    
\bes
&&\left\| \sum_{y \in V} P(t-s_1, x, y) v^p(s_1, y)\mu(y)-\sum_{y \in V} P(t-s_2, x, y) v^p(s_2, y)\mu(y) \right\|_{\ell^{s_1}_{x}(V)}\nm\\
&\le&\left\|I(t,s_1,s_2,\cdot)\right\|_{\ell^{s_1}(V)}+\left\|II(t,s_1,s_2,\cdot)\right\|_{\ell^{s_1}(V)}.\label{6.17}
\ees
    Through discussions similar to \eqref{6.2}-\eqref{6.7}, one may deduce that there exists $\hat{C}_1>0$ so that
    \bes
\|I(t,s_1,s_2,\cdot)\|_{\ell^{s_1}(V)}&=&|s_2-s_1|\left\|\sum_{y\in V}P(\theta,x,y)\Delta v^p(s_1,y)\mu(y)\right\|_{\ell^{s_1}(V)}\nm\\
    &\le&|s_2-s_1|\left\|\Delta v^p(s_1, \cdot)\right\|_{\ell^{s_1}(V)}\nm\\
    &\le&|s_2-s_1| D_\mu\left(1+d_m\right)\left\|v^p(s_1,\cdot)\right\|_{\ell^{s_1}(V)}\nm\\
    &=&|s_2-s_1| D_\mu{ }\left(1+d_m\right)\|v(s_1,\cdot)\|_{\ell^{ps_1}(V)}^p\nm\\
    &\le&\hat{C}_{1}|s_2-s_1| D_\mu\left(1+d_m\right)\left\|{v(s_1, \cdot)}\right\|_{\ell^{s_1}(V)}^p\nm\\
    &\le&\hat{C}_{1}|s_2-s_1| D_\mu\left(1+d_m\right)\|v\|_{C\left([0,T];\ell^{s_1}(V)\right)}^p\nm\\
    &=:&|s_2-s_1| C_{15},\label{6.12}
    \ees
    where $\theta>0$ is between $t-s_1$ and $t-s_2$, which is independent of $x$ and $y$, and $$C_{15}=\hat{C}_{1}D_\mu\left\|v\right\|_{C([0,T];\ell^{s_1}(V))}^p\left(1+d_m\right).$$
By \eqref{6.15.} and Lemma \ref{l3.9}, arguing as in \eqref{6.16}, one has there exists $C_{2}^{*}>0$ so that
    \bes
\|II(t,s_1,s_2,\cdot)\|_{\ell^{s_1}(V)}&\leq& \left\|v^p(s_1, \cdot)-v^p(s_2, \cdot)\right\|_{\ell^{s_1}(V)}\nm\\
&\le&C_{2}^{*} \left\|v(s_1, \cdot)-v(s_2, \cdot)\right\|_{\ell^{s_1}(V)}.
\label{6.21}
    \ees
It follows from \eqref{6.17}, \eqref{6.12}, \eqref{6.21} and \eqref{6.15.}  
   that for any $t\in(0,T]$, 
    \bess
    \sum\limits_{y\in V}P(t-s,\cdot,y)v^p(s,y)\mu(y)\text{ is continuous w.r.t.~} s\in[0,t] \text{ in } \ell^{s_1}(V).
    \eess
Thus, we know that for any $t\in(0,T]$,
\bes 
\sum\limits_{y\in V}P(t-s,\cdot,y)v^p(s,y)\mu(y)\text{ is strongly measurable w.r.t.~} s\in[0,t] \text{ in } \ell^{s_1}(V).\label{6.26..}
\ees

{\bf Step 5}. For $t_1,t_2\in[0,T]$ and $x\in V$. Without loss of generality, we assume $t_2<t_1$.
Define
    \bes
M(t,x):=\int_0^{t_1}\sum_{y\in V}P(t_1-s,x,y)v^p(s,y)\mu(y)ds-\int_0^{t_2} \sum_{y \in V} P(t_2-s,x,y)v^p(s,y)\mu(y)ds.\nm
    \ees
Hence, by \eqref{6.26..} and Lemma \ref{t2.1}, we have
    \bes
    &~&\left\| M(t,\cdot)\right\|_{\ell^{s_1}(V)}\nm\\
    &\le&
    \left\|\int_0^{t_2}\sum_{y \in V} P(t_1-s,\cdot, y) v^p(s, y) \mu(y)-\sum_{y \in V} P(t_2-s,\cdot, y) v^p(s, y) \mu(y)ds\right.\nm\\
    &&\left.+\int_{t_2}^{t_1}\sum_{y \in V} P(t_1-s,\cdot,y) v^p(s,y) \mu(y)ds\right\|_{\ell^{s_1}(V)}\nm\\
&\le&\int_0^{t_2}\left\|\sum_{y \in V} P(t_1-s,\cdot,y)v^p(s,y)\mu(y)-\sum_{y\in V}P(t_2-s,\cdot, y) v^p(s, y) \mu(y)\right\|_{\ell^{s_1}(V)}ds\nm\\
    &&+\int_{t_2}^{t_1}\left\| \sum_{y \in V} P(t_1-s,\cdot, y) v^p(s, y) \mu(y)\right\|_{\ell^{s_1}(V)}ds.\label{6.15}
    \ees
    By similar arguments as in the proof of \eqref{6.12}, there exists $C_{16}>0$ so that
    \bes
    &&\left\|\sum_{y \in V} P(t_1-s,\cdot, y) v^p(s, y) \mu(y)-\sum_{y \in V} P(t_2-s,\cdot, y) v^p(s, y) \mu(y)\right\|_{\ell^{s_1}(V)}\nm\\
    &\le&(t_1-t_2) C_{16}.\label{6.16.}
    \ees
    By \eqref{6.15}, \eqref{6.16.} and Lemma \ref{l3.9}, we see that 
    \bess
    \left\| M(t,\cdot)\right\|_{\ell^{s_1}(V)}&\le&\int_0^{t_2}(t_1-t_2) C_{16}ds+\int_{t_2}^{t_1}\left\|v^{p}(s,\cdot)\right\|_{\ell^{s_1}(V)} ds\nm\\
    &\le&(t_1-t_2)TC_{16}+\int_{t_2}^{t_1}\left\|v(s,\cdot)\right\|^{p}_{\ell^{ps_1}(V)}ds.
    \eess
Thanks to $ps_1\ge s_1$, by Lemma \ref{l3.5}, it follows that there exists $C_{*}>0$ so that
    \bess
    \left\| M(t,\cdot)\right\|_{\ell^{s_1}(V)}&\le&(t_1-t_2) TC_{16}+\int_{t_2}^{t_1}C_{*}\left\|v(s,\cdot)\right\|^{p}_{\ell^{s_1}(V)}ds\nm\\
    &\le&(t_1-t_2) TC_{16}+(t_1-t_2)C_{*}\left\|v\right\|^{p}_{C\left([0,T];\ell^{s_1}(V)\right)}.
    \eess
It follows that 
\bes\label{6.35}
\int_0^t \sum\limits_{y \in V} P(t-s,\cdot,y)v^p(s,y)\mu(y)ds \text{ is continuous w.r.t.}~t\in[0,T]~\text{in~} \ell^{s_{1}}(V).
\ees

{\bf Step 7}. Due to $s_1 p \ge s_1$, in view of \eqref{6.26..}, by Lemma \ref{l3.5} and Theorem \ref{t2.1}, we know that there exists $\hat{C}_{16}>0$ so that
\bess
    \left\|\int_0^t \sum_{y \in V} P(t-s, \cdot, y) v^p(s, y) \mu(y) d s\right\|_{\ell^{s_1}(V)} &\leq& \int_0^t\left\|\sum_{y \in V} P\left(t-s,\cdot,y\right) v^p(s, y)\mu(y)\right\|_{\ell^{s_1}(V)} ds\nm\\
    &\leq&\int_0^t\left\|v^p(s,\cdot)\right\|_{\ell^{s_1}(V)} ds\nm\\
    &\leq& \int_0^t\|v(s,\cdot)\|_{\ell^{s_1p}(V)}^p ds\nm\\
    &\leq&\hat{C}_{16} \int_0^t\left\|v\left(s,\cdot\right)\right\|_{\ell^{s_1}(V)}^p ds\nm\\
    &\leq&\hat{C}_{16}\|v\|_{C\left([0, T];\ell^{s_1}(V)\right) }^pT<+\infty.
    \eess
It follows that 
\bes
\max_{t\in[0,T]}\left\|\int_0^t\sum_{y\in V}P(t-s,\cdot,y)v^p(s,y)\mu(y)ds\right\|_{\ell^{s_1}(V)}\leq+\infty.
\ees
Combining this with \eqref{6.35}, we conclude 
    \begin{equation}\label{6.37.}
        J(t,\cdot):=\int_{0}^{t}\sum_{y \in V} P(t-s,\cdot,y)v^p(s,y)\mu(y)ds\in C\left([0,T];\ell^{s_1}(V)\right) \text {. }
    \end{equation}

By \eqref{3.18}, \eqref{6.11.} and \eqref{6.37.}, we see that 
    \begin{equation}\label{6.38}
        \Phi_1(v) \in C\left([0, T];\ell^{s_1}(V)\right).
    \end{equation}
Similarly, since $s_1\ge r_2\ge 1$, we could show that 
    \begin{equation}\label{6.39}
        \Phi_2(u) \in C\left([0, T];\ell^{s_1}(V)\right).
    \end{equation}
By \eqref{6.2.}, \eqref{6.38} and \eqref{6.39}, we have 
    \begin{equation*}
        \|\psi(u, v)\| \le\left\|\Phi_1(v)\right\|_{C\left([0,T];\ell^{s_1}(V)\right)}+\left\|\Phi_2(u)\right\|_{C\left([0, T] ; \ell^{s_1}(V)\right)}.
    \end{equation*}
By similar arguments as in the proof of \eqref{3.40.}, one may see that  
\begin{equation*}\label{}
    0 \leq \Phi_1(v)(t, x), \Phi_2(u)(t, x) \leq M \text { for } t \in[0, T] \text { and } x \in V \text {. }
\end{equation*}
Therefore, we know that 
    \begin{equation*}
        \psi: A_T\to A_T.
    \end{equation*}

{\bf Step 8}. By \eqref{6.26..}, we know that for all $t\in(0,T]$,
$\sum\limits_{y \in V} P(t-s, \cdot, y)\left[v_1^p(s, y)-v_2^p(s, y)\right] \mu(y)$ is strongly measurable w.r.t. $s \in[0, t]$ in $\ell^{s_1}(V)$. Similarly, $\sum\limits_{y \in V} P(t-s, \cdot, y)\left[u_1^q(s, y)-u_2^q(s, y)\right] \mu(y)$ is strongly measurable w.r.t. $s \in[0, t]$ in $\ell^{s_1}(V)$. Hence, by Lemma \ref{t2.1}, one could conclude that
    \bes
    &&\left\|\psi\left(u_1, v_1\right)-\psi\left(u_2, v_2\right)\right\|\nm\\
    & \le& \max _{t \in[0, T]}\left\lbrace  \int_0^t\left\|\sum_{y \in V} P(t-s,\cdot, y)\left|v_1^p(s, y)-v_2^p(s, y)\right| \mu(y)\right\|_{\ell^{s_1}(V)} d s \right.\nm\\
    &&\left.\quad+\int_0^t\left\|\sum_{y \in V} P\left(t-s,\cdot,y\right)\left|u_1^q(s, y)-u_2^q(s, y)\right| \mu(y)\right\|_{\ell^{s_1}(V)}ds\right\}.\label{6.40}
    \ees
    From \eqref{A}, for $p>1$, arguing as in \eqref{6.16}, we may deduce that
    \bes
    \left|v_1^p(s, y)-v_2^p(s, y)\right| 
    &\leq&p\left(\left\|v_1\right\|_{C\left([0,T];\ell^{\infty}(V)\right) }+\left\|v_2\right\|_{C\left([0,T];\ell^{\infty}(V)\right)}\right)^{p-1}\left|\left(v_1-v_2\right)(s, y)\right|\nm\\
    &=:&\tilde{C}_{15}\left|\left(v_1-v_2\right)(s,y)\right|,\label{6.23}
    \ees
where
$$\tilde{C}_{15}:=\max\left\{1,p\left(\left\|v_1\right\|_{C\left([0,T];\ell^{\infty}(V)\right) }+\left\|v_2\right\|_{C\left([0,T];\ell^{\infty}(V)\right)}\right)^{p-1}\right\}>0.$$
Similarly, thanks to $q\ge1$, we have
    \bes
    \left|u_1^q(s, y)-u_2^q(s, y)\right| 
    \le C_{16}\left|\left(u_1-u_2\right)(s, y)\right|,\label{6.24}
    \ees
where $$C_{16}:=\max\left\lbrace1,q\left(\left\|v_1\right\|_{C\left([0,T];\ell^{\infty}(V)\right) }+\left\|v_2\right\|_{C\left([0,T];\ell^{\infty}(V)\right)}\right)^{q-1}\right\rbrace>0.$$
    Furthermore, recalling that $s_1 \ge 1$, by \eqref{6.23}, \eqref{6.24} and Lemma 3.7, we know that for $t \in[0, T], s \in[0, t]$,
    \bes
    &&\left\|\sum_{y \in V} P(t-s, \cdot, y)\left|v_1^p(s, y)-v_2^p(s, y)\right| \mu(y)\right\|_{\ell^{s_1}(V)}\nm\\
    &\le&\tilde{C}_{15}\left\|\sum_{y \in V} P\left(t-s,\cdot, y\right)\left|v_1(s, y)-v_2(s, y)\right| \mu(y)\right\|_{\ell^{s_1}(V)}\nm\\
    &\le&\tilde{C}_{15}\left\|v_1\left(s,\cdot\right)-v_2\left(s,\cdot\right)\right\|_{\ell^{s_1}(V)}\label{6.25}
    \ees
and 
    \bes
    &&\left\|\sum_{y \in V} P(t-s,\cdot,y)\left|u_1^q(s, y)-u_2^q(s, y)\right| \mu(y)\right\|_{\ell^{s_{1}}(V)}\nm\\
    &\le& C_{16}\left\|\sum_{y \in V} P(t-s,\cdot,y)\left|u_1(s, y)-u_2(s, y)\right| \mu(y)\right\|_{\ell^{s_1}(V)}\nm\\
    &\le& C_{16}\left\|u_1\left(s,\cdot\right)-u_2\left(s, \cdot\right)\right\|_{\ell^{s_1}(V)}.\label{6.26}
    \ees
    Combining \eqref{6.40}, \eqref{6.25} and \eqref{6.26}, one may see that 
    \bess
    &&\|\psi\left(u_1,v_1\right)- \psi\left(u_2, v_2\right)\|\nm\\
    &\leq&\max_{t \in[0, T]}\left\{\int_0^t \tilde{C}_{15}\left\|v_1(s, \cdot)-v_2\left(s,\cdot\right)\right\|_{\ell^{s_1}(V)} ds+\int_0^t C_{16}\left\|u_1(s,\cdot)-u_2(s, \cdot)\right\|_{\ell^{s_1}(V)}ds\right\}\nm\\
    &\leq&T\max\left\lbrace \tilde{C}_{15},C_{16}\right\rbrace \left\{\left\|v_1-v_2\right\|_{C\left([0, T] ;\ell^{s_1}(V)\right)}+\left\|u_1-u_2\right\|_{C\left([0,T];\ell^{s_1}(V)\right)}\right\}\nm\\
    &:=&C_{17}(T)\|\left(u_1,v_1\right)-\left(u_2, v_2\right)\|,
\eess
where $C_{17}(T):=T\max\left\lbrace \tilde{C}_{15},C_{16}\right\rbrace$. Thus we see that $\psi$ is a strict contraction map if $T$ is sufficiently small so that $ C_{17}(T)<1$.

{\bf Step 9}. Select $T_1\in(0,T]$ so small that $C_{17}(T_1)<1$. We can then apply Banach's Fixed Point Theorem to find a fixed point $(u,v)$ of $\psi$ and $$u, v\in C\left([0,T_1];\ell^{s_1}(V)\right).$$ By Lemma \ref{l3.4}, this implies that $(u,v)$ satisfies \eqref{3.26.} and $$u, v\in C\left([0,T_1];\ell^{\infty}(V)\right).$$ 
From this, through similar discussions as at the end of the first step in the proof of Theorem \ref{t2.2}, we know $(u,v)$ is a solution to \eqref{1.4.} for $t\in[0,T_1]$ and $x\in V$. Recalling that $p,q\ge 1$, by the proof Theorem \ref{t2.2}, we see $(u,v)\equiv(\tilde{u},\tilde{v})$ for $t\in\left[0,T_1\right]$ and $x\in V$.
    
    We now complete the proof.
\end{proof}

\begin{lem}\label{l6.2}
Let $G=(V,E,\omega,\mu)$ be a graph  satisfying CDE$(n,0)$ and the condition {\bf (LVG$m$)}. Assume $$ \mu_{\min}>0,~ \mu_{\max}<\infty,~ \omega_{\min}>0 \text{ and } D_{\mu}<\infty.$$ Suppose $p\le q, pq>1, p,q\ge 1$, and $\frac{q+1}{p q-1}<\frac{m}{2}$. 
Let $T_{3}>0$, 
\bes
	r_1=\frac{m}{2}\left(\frac{p q-1}{p+1}\right)\text{ and }r_2=\frac{m}{2}\left(\frac{pq-1}{q+1}\right).\label{7.1}
	\ees
Suppose also that $u_0\in \ell^{r_2}(V), v_0\in \ell^{r_2}(V)$, $(u(t,x),v(t,x))$ is a solution to \eqref{1.4.} for $t\in[0,T_{3}]$ and $x\in V$. Then there exist constant $\hat{C}_{1}>0$ and $\delta\in(0,1)$  so that if $$\left\|u_0\right\|_{\ell^{r_{2}}(V)}+\left\|v_0\right\|^{p}_{\ell^{r_2}(V)}\le\hat{C}_{1},$$ then for any $T>0$,
	\bess
	u(t,x)\in C\left([0,T];\ell^{s_1}(V)\right),~v(t,x)\in C\left([0,T];\ell^{s_2}(V)\right),
	\eess
where $s_1=\frac{r_1}{\delta}$ and $s_2=\frac{r_2}{\delta}$.
\end{lem}

\begin{proof}
{\bf Step 1}. Since $p\le q$, $\frac{q+1}{pq-1}<\frac{m}{2}$, we have 
\bes
1\le r_2\le r_1.\label{7.2}
\ees
Due to $pq>1, \frac{q(p+1)}{q+1}>\frac{1+q}{q+1}=1$. From $\frac{q+1}{p q-1}<\frac{m}{2}$, there exists a small $\epsilon_1>0$ such that \bes
        \frac{2(q+1)(1+\tilde{\epsilon})}{m(p q-1)}<1
    \ees
and 
\bes
\frac{q^2(p+1)^2}{(q+1)^2} \ge 1+\tilde{\epsilon}\label{6.42}
\ees
for all $\tilde{\epsilon} \in\left(0, \epsilon_1\right)$. Thanks to $pq>1$, 
\bess
\frac{q+1}{q(p+1)}<\frac{q+1}{1+q}=1.
\eess
Since $q>-1$ and $\frac{q+1}{pq-1}<\frac{m}{2}$, we have $$\frac{q+1}{q(p+1)}<\frac{q+1}{p q-1}<\frac{m}{2}.$$ Thus, there exists $0<\epsilon_2$ such that
\bes
\frac{(q+1)(1+\epsilon)}{q(p+1)}<\min \left\{1, \frac{m}{2}\right\} \text { for all } \epsilon \in\left(0, \epsilon_2\right) \text {. }
\ees
Therefore, we can find $\epsilon \in\left(0, \min \left(\epsilon_1, \epsilon_2\right)\right)$ so that
\bes
&&\frac{2(q+1)(1+\epsilon)}{m(pq-1)}<1,\label{6.27}\\[2mm]
&&\frac{q^2(p+1)^2}{(q+1)^2} \ge 1+\epsilon,\label{7.6}
\ees
and
\bes
    \frac{(q+1)(1+\epsilon)}{q(p+1)}<\min \left\{1, \frac{m}{2}\right\}.\label{6.47}
\ees

{\bf Step 2}. Let 
\begin{equation}\label{7.9}
\delta \in\left(\frac{q+1}{q(p+1)}, \frac{(q+1)(1+\epsilon)}{q(p+1)}\right);
\end{equation}
then 
\bes
0<\delta<\min \left\{1, \frac{m}{2}\right\}.\label{7.10}
\ees
Taking
\bes
s_1=\frac{m}{2 \delta}\left(\frac{p q-1}{p+1}\right)\label{7.11}
\ees
and
\bes
s_2=\frac{m}{2 \delta}\left(\frac{p q-1}{q+1}\right).\label{7.12}
\ees
Then by \eqref{7.2}, we have
\bes
\frac{r_1}{\delta}=s_1\ge s_2=\frac{r_2}{\delta}>1.\label{7.13}
\ees
Due to $\delta\in(0,1)$, by \eqref{7.12}, 
\bes
\frac{r_2}{\delta}=s_2> r_2.\label{7.14}
\ees
In view of this, recalling that $u_0,v_0\in\ell^{r_2}(V)$, by \eqref{7.14} and Lemma \ref{l6.1}, we know there exists $T_{4}\in (0,T_3]$ so that 
\bes
u(t,x)\in C\left([0,T_{4}];\ell^{s_2}(V)\right),~v(t,x)\in C\left([0,T_{4}];\ell^{s_2}(V)\right).\label{7.15}
\ees
By \eqref{7.13} and Lemma \ref{l3.5}, we know that 
\bess
C\left([0,T_{4}];\ell^{s_2}(V)\right)\subset C\left([0,T_{4}];\ell^{s_1}(V)\right).
\eess
Combining this and \eqref{7.15} yields that
\bes
u(t,x)\in C\left([0,T_{4}];\ell^{s_1}(V)\right),~v(t,x)\in C\left([0,T_{4}];\ell^{s_2}(V)\right).\label{7.16}
\ees

{\bf Step 3}. We next show that
\bes
s_2\ge\frac{s_1}{q},\label{6.51}
\ees
and
\bes
s_2>p.\label{7.17}
\ees
Recalling $0<\delta<1$. It is easy to check that
    $$
    \mathbf{s}_2\ge\frac{s_1}{q} \Leftrightarrow \frac{r_2}{\delta}\ge\frac{s_1}{q}\Leftrightarrow\frac{r_2}{s_1}\ge \frac{\delta}{q} \Leftrightarrow\frac{p+1}{q+1}\delta\ge\frac{\delta}{q}\Leftrightarrow\frac{q(p+1)}{q+1}\ge1\Leftrightarrow pq\ge1.
    $$
Thus we get $s_2 \ge \frac{s_1}{q}$. We claim that $s_2>p$. It is easy to see that
    $$
    s_2>p \Leftrightarrow \frac{r_2}{\delta}>p\Leftrightarrow \frac{r_2}{p}>\delta\Leftrightarrow \frac{m}{2}\left(\frac{pq-1}{q+1}\right)\frac{1}{p}>\delta.
    $$
Recalling that $\delta<\frac{(q+1)(1+\epsilon)}{q(p+1)}$, by \eqref{7.9}, it suffices to prove that
$$
    \frac{(q+1)(1+\epsilon)}{q(p+1)} \leq \frac{m}{2} \frac{p q-1}{q+1} \frac{1}{p},
$$
which is equivalent to
$$
\frac{2(1+\epsilon)(q+1)}{m(p q-1)} \le \frac{(p+1) q}{(q+1) p} \text {. }
$$
Due to $q\ge p$, by \eqref{6.27}, we obtain
$$
\frac{2(1+\epsilon)(q+1)}{m(p q-1)}<1=\frac{p q+p}{(q+1) p}\le\frac{(p+1)q}{(q+1)p} .
$$
Hence, we get $s_2>p$.

{\bf Step 4}. Now, we show that 
\bes
s_1> r_1,\label{7.19.}
\ees  
\bes
s_1\ge\frac{s_2}{p},\label{7.20}
\ees  
and 
\bes
s_1>q. \label{7.19}
\ees
By \eqref{7.10}, \eqref{7.13}, we have \eqref{7.19.}. Obviously, 
$$s_1 \ge \frac{s_2}{p} \Leftrightarrow \frac{r_1}{\delta} \ge \frac{s_2}{p} \Leftrightarrow \frac{r_1}{s_2} p \ge \delta.$$ Recalling $p\le q$, it follows from \eqref{7.1} and \eqref{7.12} that 
\bes
    \frac{r_1}{s_2} p=\frac{\frac{m}{2}\left(\frac{p q-1}{p+1}\right)}{\frac{m}{2 \delta}\left(\frac{p q-1}{q+1}\right)} p=\delta \frac{q+1}{p+1}p\ge\delta.\ees
Thus, we obtain $\frac{r_1}{s_2}p\ge\delta$ and hence \eqref{7.20}.
Obviously,
$$
s_1 > q \Leftrightarrow\frac{m}{2 \delta} \frac{p q-1}{p+1}> q \Leftrightarrow \frac{m}{2 \delta} \frac{p q-1}{q(p+1)}> 1 \Leftrightarrow \frac{m}{2} \frac{p q-1}{q(p+1)}>\delta.
$$
Thus, by virtue of \eqref{7.9}, to prove $s_{1}>q$, it is sufficient to show that
\bes
\frac{m}{2} \frac{p q-1}{q(p+1)} \ge \frac{(q+1)(1+\epsilon)}{q(p+1)} \text {, }
\ees
which is equivalent to
$$
\frac{m}{2} \frac{p q-1}{q+1} \ge 1+\epsilon \text {. }
$$
This can be obtained by \eqref{6.27}. Hence, we obtain $$s_1>q.$$ 

{\bf Step 5}. In view of \eqref{7.11} and \eqref{7.12}, recalling that $\delta>0$, we have
\bes
\frac{1}{s_1}>\frac{p}{s_2}-\frac{2}{m} &\Leftrightarrow&\frac{2 \delta}{m} \frac{p+1}{p q-1}>\frac{2 \delta}{m} \frac{q+1}{pq-1}p-\frac{2}{m}\nm\\[2mm]
&\Leftrightarrow&\frac{2}{m}>\frac{2\delta}{m}\frac{pq-1}{pq-1}\nm\\[2mm]
&\Leftrightarrow&\delta<1.\label{6.37}
\ees
Hence, by \eqref{7.10}, we obtain 
\bes
\frac{1}{s_1}>\frac{p}{s_2}-\frac{2}{m}.\label{7.21}
\ees
To prove $\frac{1}{s_2}>\frac{q}{s_1}-\frac{2}{m}$, by \eqref{7.11} and \eqref{7.12},
it is sufficient to show that
\bes
\frac{2\delta}{m}\left(\frac{q+1}{pq-1}\right)>q\frac{2\delta}{m}\frac{p+1}{pq-1}-\frac{2}{m}.\label{7.22}
\ees
By \eqref{7.10}, we have
$$
\frac{2}{m}>q \frac{2 \delta}{m} \frac{p+1}{p q-1}-\frac{2 \delta}{m} \frac{q+1}{p q-1}=\frac{2 \delta}{m} \frac{p q-1}{p q-1} .
$$
It follows that \eqref{7.22} holds. Thus we get
$$
\frac{1}{s_2}>\frac{q}{s_1}-\frac{2}{m}.
$$

{\bf Step 6}. In view of \eqref{7.11} and \eqref{7.12}, by direct calculations, it can be concluded that
$$
\begin{aligned}
& \frac{p}{s_2}-\frac{1}{s_1}=\frac{2 \delta}{m} \frac{q+1}{p q-1} p-\frac{2 \delta}{m} \frac{p+1}{p q-1}=\frac{2 \delta}{m} \frac{p q-1}{p q-1}=\frac{2 \delta}{m}, \\[2mm]
& \frac{q}{s_1}-\frac{1}{s_2}=\frac{2 \delta}{m} \frac{p+1}{p q-1} q-\frac{2 \delta}{m} \frac{q+1}{p q-1}=\frac{2 \delta}{m} \frac{p q-1}{p q-1}=\frac{2 \delta}{m}.
\end{aligned}
$$
It follows that
\bes
&1+\frac{1}{s_1}=\frac{1}{\frac{s_2}{p}}+1-\frac{2\delta}{m},\label{7.24}
\ees
\bes
&1+\frac{1}{s_2}=\frac{1}{\frac{s_1}{q}}+1-\frac{2\delta}{m}.\label{6.74}
\ees
Let 
\bes 
k_*=\frac{1}{1-\frac{2 \delta}{m}}.\label{6.55}
\ees 
Then by \eqref{7.10}, we have 
\bes
k_{*}\in[1,+\infty).\label{7.28.}
\ees

{\bf Step 7}. Next we always assume $t\in(0,T_{4}]$. 
Thanks to $s_1\ge 1$, by virtue of \eqref{7.16}, \eqref{7.13}, \eqref{7.17}, \eqref{7.28.} and \eqref{7.24}, using Theorem \ref{t2.1.}, we know for $s\in[0,t]$,
\bes
&&\left\|\sum_{y \in V} P(t-s, \cdot, y) v^p(s, y) \mu(y)\right\|_{\ell^{s_1}(V)}\nm\\
&\leq& C_4\left\|v^p(s, \cdot)\right\|_{\ell^{s_2 / p}(V)}(t-s)^{-\frac{m}{2}\left(1-\frac{1}{k_{*}}\right)}\nm\\
&=& C_4\|v(s, \cdot)\|_{\ell^{s_2}(V)}^p(t-s)^{-\frac{m}{2}\left(1-\frac{1}{k_{*}}\right)}.\label{7.35}
\ees
In view of \eqref{7.13}, \eqref{7.19}, \eqref{7.28.}, \eqref{7.19} and \eqref{6.74}, by Theorem \ref{t2.1.}, we know for $s\in[0,t]$,
\bes
&&\left\|\sum_{y \in V} P(t-s, \cdot, y) u^q(s, y) \mu(y)\right\|_{\ell^{s_2}(V)}\nm\\
&\leq& C_4\left\|u^q(s, \cdot)\right\|_{\ell^{s_1 / q}(V)}(t-s)^{-\frac{m}{2}\left(1-\frac{1}{k_{*}}\right)}\nm\\
&=& C_4\|u(s, \cdot)\|_{\ell^{s_1}(V)}^q(t-s)^{-\frac{m}{2}\left(1-\frac{1}{k_{*}}\right)}.\label{6.78}
\ees
Let $m_{1}$ satisfy
\bes
1+\frac{1}{s_1}=\frac{1}{r_1}+\frac{1}{m_1}.\label{7.30}
\ees
Then by \eqref{7.19.}, we know $m_1>1$. From this, \eqref{7.2}, \eqref{7.13} and \eqref{7.30}, using Theorem \ref{t2.1.}, we see 
\bes
\left\|\sum_{y\in V}P(t, x, y) u_0(y)\mu(y)\right\|_{\ell^{s_1}(V)}\le C_4\left\|u_0\right\|_{\ell^{r_1}(V)} t^{-\frac{m}{2}\left(\frac{1}{r_1}-\frac{1}{s_1}\right)}.\label{7.32}
\ees
Let $m_{2}$ satisfy
\bes
1+\frac{1}{s_2}=\frac{1}{r_2}+\frac{1}{m_2}.\label{7.33}
\ees
Then by \eqref{7.14}, we know $m_2>1$. From this, \eqref{7.2}, \eqref{7.13} and \eqref{7.33}, using Theorem \ref{t2.1.}, we see 
\bes
\left\|\sum_{y\in V}P(t, x, y) v_0(y)\mu(y)\right\|_{\ell^{s_2}(V)}\le C_4\left\|v_0\right\|_{\ell^{r_2}(V)} t^{-\frac{m}{2}\left(\frac{1}{r_2}-\frac{1}{s_2}\right)}.\label{7.34}
\ees 
By Lemma \ref{l3.12}, we know that 
\bes
(u,v) \text{ satisfies } \eqref{3.26.}.\label{7.37}
\ees
In view of \eqref{7.16}, by similar arguments as the proof of \eqref{6.26..}, we know 
\bes 
\sum\limits_{y\in V}P(t-s,\cdot,y)v^p(s,y)\mu(y)\text{ is strongly measurable w.r.t.~} s\in[0,t] \text{ in } \ell^{s_1}(V).\label{7.38}
\ees
 Hence, by triangle inequality, \eqref{7.37}, \eqref{7.32}, \eqref{7.34}, \eqref{7.35}, \eqref{6.78} and \eqref{7.38}, using Theorem \ref{t2.1.}, we deduce that
\bes
\|u(t,\cdot)\|_{\ell^{s_1}(V)}&\le&\left\|\sum_{y\in V}P(t,\cdot,y) u_0(y)\mu(y)\right\|_{\ell^{s_1}(V)}+\int_{0}^{t}\left\|\sum_{y\in V}P(t-s,\cdot,y)v^p(s,y) \mu(y)\right\|_{\ell^{s_1(V)}}ds\nm\\
&\le&C_4\left\|u_0\right\|_{\ell^{r_1}(V)} t^{-\frac{m}{2}\left(\frac{1}{r_1}-\frac{1}{s_1}\right)}+C_{4}\int_0^t(t-s)^{-\frac{m}{2}\left(1-\frac{1}{k_*}\right)}\|v(s, \cdot)\|_{\ell^{s_2}(V)}^pds, \label{7.39}
\ees 
and
\bes
\|v(t,\cdot)\|_{\ell^{s_2}(V)}&\le&\left\|\sum_{y\in V}P(t,x,y) v_0(y)\mu(y)\right\|_{\ell^{s_2}(V)}+\int_{0}^{t}\left\|\sum_{y\in V}P(t-s,\cdot,y)u^q(s,y) \mu(y)\right\|_{\ell^{s_1(V)}}ds\nm\\
&\le&C_4\left\|v_0\right\|_{\ell^{r_2}(V)} t^{-\frac{m}{2}\left(\frac{1}{r_2}-\frac{1}{s_2}\right)}
+C_{4}\int_0^t(t-s)^{-\frac{m}{2}\left(1-\frac{1}{k_*}\right)}\|u(s, \cdot)\|_{\ell^{s_1}(V)}^qds.\label{7.40}
\ees
By \eqref{6.55}, we have
\bes
 \frac{m}{2}\left(1-\frac{1}{k_*}\right)=\delta.\label{6.56}
\ees 
Then substituting the second inequality into the first inequality above, we know that
\bes
\|u(t,\cdot)\|_{\ell^{s_1}(V)}&\le&C_4\left\|u_0\right\|_{\ell^{r_1}(V)} t^{-\frac{m}{2}\left(\frac{1}{r_1}-\frac{1}{s_1}\right)}+C_{4}\int_0^t(t-s)^{-\delta}\nm\\[2mm]
&&\left[C_4\left\|v_0\right\|_{\ell^{r_2}(V)} s^{-\frac{m}{2}\left(\frac{1}{r_2}-\frac{1}{s_2}\right)}+C_{4}\int_0^s(s-\tau)^{-\delta}\|u(\tau, \cdot)\|_{\ell^{s_1}(V)}^q d \tau\right]^p ds\nm\\[2mm]
&=&C_4^{}\left\|u_0\right\|_{\ell^{r_1}(V)} t^{-\frac{m}{2}\left(\frac{1}{r_1}-\frac{1}{s_1}\right)}\nm\\[2mm]
&&\quad+2^{p-1} C_4^{1+p}\left\|v_0\right\|_{\ell^{r_2}(V)}^p\left[\int_0^t(t-s)^{-\delta} s^{-\frac{m}{2}\left(\frac{1}{r_2}-\frac{1}{s_2}\right) p} d s\right]\nm\\[2mm]
&&\quad+2^{p-1}C_{4}^{1+p}\int_0^t(t-s)^{-\delta}\left[ \int_0^s(s-\tau)^{-\delta}\|u(\tau, \cdot)\|_{\ell^{s_1}(V)}^q d \tau\right]^p ds.\label{6.57}
\ees
It is well-known that
\bes
\int_0^t(t-s)^{-\alpha} s^{-\beta} ds=t^{1-\alpha-\beta} B(1-\beta, 1-\alpha),~\alpha, \beta \in(0,1),\label{6.58}
\ees
where $B$ is the beta function.
Let 
\bes w=\frac{m}{2}\left(\frac{1}{r_1}-\frac{1}{s_1}\right).\label{6.59}
\ees 
We now show that 
\bes 
wq,~\left(wq+\delta-1\right)p\in(0,1).\label{7.45.}
\ees
By \eqref{6.59}, we know 
\bes
wq=\frac{m}{2}\left(\frac{1}{r_1}-\frac{1}{s_1}\right)q.\label{7.46}
\ees
Thanks to $\frac{mq}{2}>0$, by \eqref{7.46} and \eqref{7.19.}, we know 
 \bes
 wq>0.\label{7.45}
 \ees 
 On the other hand, by \eqref{7.46}, \eqref{7.1} and \eqref{7.11}, we have
\bes
wq<1&\Leftrightarrow&\frac{m}{2}\left(\frac{1}{r_1}-\frac{\delta}{r_1}\right)q<1\Leftrightarrow \frac{m}{2} \frac{1-\delta}{r_1} q<1 \Leftrightarrow \frac{m}{2} \frac{1-\delta}{\frac{m}{2} \frac{p q-1}{p+1}} q<1 \Leftrightarrow(1-\delta) \frac{p+1}{p q-1} q<1\nm\\
&& \Leftrightarrow 1-\delta<\frac{p q-1}{p q+q} \Leftrightarrow \frac{q+1}{(p+1) q}<\delta.
\ees
Thus, by \eqref{7.9}, we know 
\bes
w q<1.\label{7.47}
\ees
By \eqref{7.46}, \eqref{7.1} and \eqref{7.11}, we see 
\bes
\left(w q+\delta-1\right)p&=&\left[\frac{m}{2}\left(\frac{1}{r_1}-\frac{1}{s_1}\right) q+\delta-1\right]p\nm\\
&=&\left[\frac{m}{2}\left(\frac{2}{m} \frac{p+1}{p q-1}-\frac{2 \delta}{m} \frac{p+1}{pq-1}\right) q+\delta-1\right]p\nm\\
&=&\left[\frac{p+1}{pq-1}(1-\delta) q+\delta-1\right]p\nm\\
&=&(1-\delta) \frac{q+1}{pq-1}p\nm\\
&=&(1-\delta) \frac{pq+p}{pq-1}.\label{7.48}
 \ees
Recalling that $\delta\in(0,1)$, $pq>1$, $p\ge 1$, by \eqref{7.48}, we deduce that
\bes
\left(w q+\delta-1\right)p>0.\label{7.49}
\ees
Recalling that $1 \le p \le q$, we have
$$
\left(\sqrt{p}+\frac{1}{\sqrt{p}}\right)^2 \le\left(\sqrt{q}+\frac{1}{\sqrt{q}}\right)^2,
$$
and thus
$$
\frac{(p+1)^2}{p} \le \frac{(q+1)^2}{q}.
$$
It follows that
$$
\frac{p+1}{p(q+1)} \le \frac{q+1}{q(p+1)}.
$$
From this and \eqref{7.9}, we see that
\begin{equation}\label{7.52}
\frac{p+1}{p(q+1)}<\delta
\end{equation}
and hence that
\begin{equation}\label{6.99}
(1-\delta)\frac{pq+p}{pq-1}<1.
\end{equation}
This implies that
\bes
\left(w q+\delta-1\right) p<1 \label{7.50}
\ees
by \eqref{7.48}. Combining \eqref{7.45}, \eqref{7.47}, \eqref{7.49} with \eqref{7.50}, we obtain \eqref{7.45.}. By \eqref{7.1} and \eqref{7.12}, we have
$$
\begin{aligned}
\frac{m}{2}\left(\frac{1}{r_2}-\frac{1}{s_2}\right)p&=\frac{m}{2}\left(\frac{1}{r_2}-\frac{1}{\frac{r_2}{\delta}}\right)p=\frac{m}{2}\left(\frac{1-\delta}{r_2}\right) p=\frac{m}{2}(1-\delta)\frac{2}{m}\frac{q+1}{p q-1}p\\&=(1-\delta)\frac{pq+p}{pq-1}.
\end{aligned}
$$
Since $\delta\in(0,1)$, $pq>1$, by \eqref{6.99}, it follows that
\begin{equation}\label{7.54}
\frac{m}{2}\left(\frac{1}{r_2}-\frac{1}{s_2}\right)p\in(0,1)\text {. }
\end{equation}
Then it follows from \eqref{6.57}, \eqref{6.58}, \eqref{7.45.}, \eqref{7.54},  \eqref{7.49} and \eqref{6.99} that 
\bes
t^w\|u(t,\cdot)\|_{\ell^{s_1}(V)}&\le& C_4\left\|u_0\right\|_{\ell^{r_1}(V)}+2^{p-1} C_4^{1+p }\left\|v_0\right\|_{\ell^{r_2}(V)}^pC_{18}t^{w+\left[1-\delta-\frac{m}{2}\left(\frac{1}{r_2}-\frac{1}{s_2}\right)p\right]}+2^{p-1}\times\nm\\[2mm]
&&C_{4}^{1+p}t^w  \int_0^t(t-s)^{-\delta}\left[\int_0^s(s-\tau)^{-\delta} \tau^{-w q}\left(\tau^w\|u(\tau,\cdot)\|_{\ell^{s_1}(V)} \right)^q d \tau\right]^p ds\nm\\[2mm]
&\le&C_4\left\|u_0\right\|_{\ell^{r_1}(V)}+2^{p-1}C_4^{1+p}\left\|v_0\right\|_{\ell^{r_2}(V)}^pC_{18}t^{w+\left[1-\delta-\frac{m}{2}\left(\frac{1}{r_2}-\frac{1}{s_2}\right) p\right]}+2^{p-1}\times\nm\\[2mm]
&&C_{4}^{1+p}t^w\int_0^t(t-s)^{-\delta}\left[\int_0^s(s-\tau)^{-\delta}\tau^{-wq}d\tau\right]^p ds\left(\sup_{s\in[0, t]}s^w\|u(s ,\cdot)\|_{\ell^{s_1}(V)}\right)^{pq}\nm\\[2mm]
&=&C_4\left\|u_0\right\|_{\ell^{r_1}(V)}+2^{p-1}C_{18}C_4^{1+p}\left\|v_0\right\|_{\ell^{r_2}(V)}^pt^{w+\left[1-\delta-\frac{m}{2}\left(\frac{1}{r_2}-\frac{1}{s_2}\right) p\right]}+2^{p-1}\times\nm\\[2mm]
&&t^wC_{4}^{1+p}t^{1-\delta+\left(1-\delta-wq\right)p}C_{19}\left(\sup_{s\in[0,t]} s^w\|u(s,\cdot)\|_{\ell^{s_1}(V)}\right)^{pq},\label{6.60}
\ees
where $C_{18}=B\left(1-\delta,1-\frac{m}{2}\left(\frac{1}{r_2}-\frac{1}{s_2}\right)p\right)$, and
$$C_{19}=B\left(1-\delta, 1+\left(1-\delta-wq\right) p\right) B^p\left(1-\delta, 1-w q\right).$$ By \eqref{7.1}, \eqref{7.13} and \eqref{6.59}, we get 
\bes
w=(1-\delta)\frac{p+1}{pq-1}.\label{7.53}
\ees
By \eqref{7.1}, \eqref{7.13} and \eqref{7.53},
\bes
1-\delta-\frac{m}{2}\left(\frac{1}{r_2}-\frac{1}{s_2}\right) p &=&1-\delta-\frac{m}{2}\frac{1-\delta}{r_2} p\nm\\
&=&1-\delta-\frac{m}{2}(1-\delta)\frac{2}{m} \frac{q+1}{pq-1}p\nm\\
&=&1-\delta-(1-\delta)\cdot \frac{pq+p}{pq-1}\nm\\
&=&(1-\delta)\frac{-1-p}{pq-1}\nm\\
&=&-(1-\delta)\frac{p+1}{pq-1}\nm\\
&=&-w.\label{7.56}
\ees
By \eqref{7.53}, we see that 
\bes
&&1-\delta+\left(1-\delta-wq\right)p\nm\\[2mm]
&=&1-\delta+\left[1-\delta-\frac{(1-\delta)(p+1)q}{pq-1}\right]p\nm\\[2mm]
&=&1-\delta+(1-\delta)\frac{(-1-q)p}{pq-1}\nm\\[2mm]
&=&-(1-\delta)\frac{p+1}{pq-1}\nm\\[2mm]
&=&-w.\label{6.105}
\ees
In view of this and \eqref{7.56}, by \eqref{6.60}, we get
\bes
t^w\|u(t,\cdot)\|_{\ell^{s_1}(V)}&\le&C_4\left\|u_0\right\|_{\ell^{r_1}(V)}+2^{p-1} C_{18}C_4^{1+p}\left\|v_0\right\|_{\ell^{r_2}(V)}^p\nm\\
&+&2^{p-1}C_{4}^{1+p}C_{19}\left(\sup_{s\in[0,t]}s^w\|u\left(s,\cdot\right)\|_{\ell^{s_1(V)}}\right)^{pq}.\label{6.63}
\ees
Define $g(t)=\sup\limits_{\tau \in[0, t]}\left(\tau^w\|u(\tau, \cdot)\|_{\ell^{s_1}(V)}\right)$, $t\in[0,T_4]$. Then due to $pq>1$, by \eqref{6.63},
\bes
g(t)\le C_{20}\left(\left\|u_0\right\|_{\ell^{r_1}(V)}+\left\|v_0\right\|_{\ell^{r_2}(V)}^p+g^{pq}(t)\right),\label{6.64}
\ees
where $C_{20}=\max \left\{C_4,2^{p-1} C_{18}C_4^{1+p},2^{p-1}C_{4}^{1+p}C_{19}\right\}.$ 
Recalling $\left(\left\|u_0\right\|_{\ell^{r_2}(V)}+\left\|v_0\right\|_{\ell^{r_2}(V)}^p\right)\le\hat{C}_1$, by \eqref{7.2} and Lemma \ref{l3.5}, we deduce that 
\bes
\left(\left\|u_0\right\|_{\ell^{r_1}(V)}+\left\|v_0\right\|_{\ell^{r_2}(V)}^p\right)\le\hat{C}_2:=\left(\mu_{min}^{\frac{r_2-r_1}{r_2r_1}}+1\right)\hat{C}_{1}.\label{7.58}
\ees
Let $\hat{C}_1$ be sufficiently small so that 
$$2^{pq}C^{pq}_{20}\hat{C}_2^{pq-1}<1.$$ 
We claim that 
\bes
g(t)<2C_{20}\hat{C}_2\text{ in }(0,T_{4}]. \label{7.60.}
\ees
 Define $g(0)=0$, then by \eqref{7.16} and the definition of $g(t)$, we deduce that $g(t)$ is nondecreasing and continuous in $t$. Suppose by way of contradiction that there exists $t_0\in(0,T_4]$ so that $g\left(t_0\right)=2C_{20}\hat{C}_2$. Then by \eqref{6.64} and \eqref{7.58}, we know that
\bes
2C_{20}\hat{C}_2 \le C_{20}\left[\hat{C}_2+\left(2C_{20}\hat{C}_{2}\right)^{pq}\right],
\ees
and hence that $1\le(2C_{20})^{pq}\hat{C}_2^{p q-1}$, which is a contradiction. Thus, we get \eqref{7.60.} and
\bes
\|u(t,\cdot)\|_{\ell^{s_1}(V)}\le t^{-w}2\hat{C}_2C_{20}.\label{7.60}
\ees
By Lemma \ref{l3.4}, this implies that
\bes
\|u(t,\cdot)\|_{\ell^{\infty}(V)}\le \mu_{min}^{-\frac{1}{s_1}}2\hat{C}_{2}C_{20}t^{-w} \text { for } t \in(0,T_4].\label{6.66}
\ees
{\bf Step 8}. Similarly, we may show that there exists a constant $C_{21}>0$ such that
\bes
\|v(t,\cdot)\|_{\ell^{\infty}(V)}\le2C_{21}t^{-w_{1}} \text { for } t \in(0,T_4),\label{7.63}
\ees
where 
\bes
w_{1}=\frac{m}{2}(\frac{1}{r_2}-\frac{1}{s_2}).\label{6.112}
\ees
Substituting \eqref{7.39} into \eqref{7.40}, we discover
\bes
\|v(t,\cdot)\|_{\ell^{s_2}(V)}&\le&C_4\left\|v_0\right\|_{\ell^{r_2}(V)} t^{-\frac{m}{2}\left(\frac{1}{r_2}-\frac{1}{s_2}\right)}+C_{4}\int_0^t(t-s)^{-\delta}\nm\\[2mm]
&&\left[C_4\left\|u_0\right\|_{\ell^{r_1}(V)} s^{-w}+C_{4}\int_0^s(s-\tau)^{-\delta}\|v(\tau, \cdot)\|_{\ell^{s_2}(V)}^pd\tau\right]^q ds\nm\\[2mm]
&=&C_4\left\|v_0\right\|_{\ell^{r_2(V)}} t^{-w_1}\nm\\[2mm]
&&\quad+2^{q-1} C_4^{q+1}\left\|u_0\right\|_{\ell^{r_1}(V)}^q\left[\int_0^t(t-s)^{-\delta}s^{-wq}d s\right]\nm\\[2mm]
&&\quad+2^{q-1}C_{4}^{q+1}\int_0^t(t-s)^{-\delta}\left[ \int_0^s(s-\tau)^{-\delta}\|v(\tau,\cdot)\|_{\ell^{s_2}(V)}^pd \tau\right]^qds.\label{7.65}
\ees
From \eqref{6.105}, \eqref{7.53} and \eqref{6.112}, we see that 
\begin{equation*}
1 - \delta - wq=\frac{-w-(1 - \delta)}{p}
=\frac{-(1 - \delta)\frac{p + 1}{pq - 1}-(1 - \delta)}{p}
=\frac{-(q + 1)(1 - \delta)}{pq - 1},
\end{equation*}
and
\begin{equation*}
\begin{aligned}
    -w_1=-\frac{m}{2}\left(\frac{1}{r_2}-\frac{1}{s_2}\right)
    =-\frac{m}{2}\left(\frac{1}{r_2}-\frac{\delta}{r_2}\right)
    =-\frac{m}{2}\frac{1 - \delta}{r_2}
   =-(1 - \delta)\frac{q + 1}{pq - 1}.
\end{aligned}
\end{equation*}
Hence, we get
\begin{equation}
1-\delta-wq=-w_1,~1-\delta-\left(\delta+w_1 p-1\right) q=-w_1.\label{7.66}
\end{equation}
Using \eqref{6.112}, \eqref{7.1} and \eqref{7.13}, we get
\begin{equation}\label{7.69}
\begin{aligned}
    \left(\delta+w_1 p-1\right) q & =\left[\delta+\frac{m}{2}\left(\frac{1}{r_2}-\frac{1}{s_2}\right) p-1\right] q=\left[\delta+\frac{m}{2}\left(\frac{1}{r_2}-\frac{\delta}{r_2}\right) p-1\right] q \\
    & =\left[\delta+\frac{m}{2} \frac{1-\delta}{r_2} p-1\right] q=\left[\delta-1+(1-\delta) \frac{m}{2} \cdot \frac{2}{m} \frac{q+1}{p q-1} p\right] q \\
    & =(1-\delta) \frac{p+1}{p q-1} q=(1-\delta) \frac{p q+q}{p q-1}
\end{aligned}
\end{equation}
Since $pq>1$ and $\delta\in(0,1)$, by \eqref{7.69}, we have
$$
(\delta+w_{1}p-1) q>0 .
$$
By \eqref{7.9}, we conclude that $\frac{q+1}{pq+q}<\delta$, and hence that $1-\delta<\frac{pq-1}{pq+q}$, and finally
$$
(1-\delta)\frac{pq+q}{pq-1}<1 \text {. }
$$
Therefore, we have 
\bes\label{7.70}
\left(\delta+w_1 p-1\right)q\in(0,1).
\ees
Obviously,
\begin{equation*}
\begin{aligned}
    0 < w_1p < 1 &\Leftrightarrow 0 < \frac{m}{2}\left(\frac{1}{r_2}-\frac{1}{s_2}\right)p < 1 \Leftrightarrow 0 < \frac{m}{2}\left(\frac{1}{r_2}-\frac{\delta}{r_2}\right)p < 1\\
    &\Leftrightarrow 0 < \frac{m}{2}\frac{1 - \delta}{r_2}p < 1 \Leftrightarrow 0 < \frac{m}{2}(1 - \delta)\frac{2}{m}\frac{q + 1}{pq - 1}p < 1\\
    &\Leftrightarrow (1 - \delta)\frac{pq + p}{pq - 1} < 1 \Leftrightarrow 0 < 1 - \delta < \frac{pq - 1}{pq + p}\\
    &\Leftrightarrow \frac{p + 1}{pq + p} < \delta < 1.
\end{aligned}
\end{equation*}
Due to  $p\leq q$, by \eqref{7.9}, we have 
$$\frac{p + 1}{pq + p} \leq \frac{q + 1}{q(p + 1)} < \delta < 1.$$
Thus we know that 
\begin{equation}\label{6.119}
0 < w_1p < 1.
\end{equation}
By \eqref{6.58}, \eqref{7.65}, \eqref{7.66}, \eqref{7.70} and \eqref{6.119}, we have
\bes
t^{w_{1}}\|v(t,\cdot)\|_{\ell^{s_2}(V)}&\le& C_4\left\|v_0\right\|_{\ell^{r_2}(V)}+2^{q-1}C_4^{q+ 1}\left\|u_0\right\|_{\ell^{r_1}(V)}^qC_{22}t^{w_{1}+\left(1-\delta-wq\right)}+2^{q-1}\times\nm\\[2mm]
&&C_{4}^{q+1}t^{w_{1}} \int_0^t(t-s)^{-\delta}\left[\int_0^s(s-\tau)^{-\delta}\tau^{-w_{1} p}\left(\tau^{w_{1}}\|v(\tau,\cdot)\|_{\ell^{s_2}(V)}\right)^pd\tau\right]^q ds\nm\\[2mm]
&\le&C_4\left\|v_0\right\|_{\ell^{r_2}(V)}+2^{q-1}C_4^{q+1}\left\|u_0\right\|_{\ell^{r_1}(V)}^qC_{22}t^{w_{1}+\left(1-\delta-w\right)q}+2^{q-1}\times\nm\\[2mm]
&&C_{4}^{q+1}t^{w_{1}}\int_0^t(t-s)^{-\delta}\left[\int_0^s(s-\tau)^{-\delta}\tau^{-w_{1}p}d\tau\right]^q ds\left(\sup_{s\in[0, t]}s^{w_{1}}\|v(s,\cdot)\|_{\ell^{s_2}(V)}\right)^{pq}\nm\\[2mm]
&=&C_4\left\|v_0\right\|_{\ell^{r_2}(V)}+2^{q-1}C_{22}C_4^{q+1}\left\|u_0\right\|_{\ell^{r_1}(V)}^q+2^{q-1}\times\nm\\[2mm]
&&C_{4}^{q+1}t^{w_1}t^{1-\delta+\left(1-\delta-w_{1}p\right)q}C_{23}\left(\sup_{s\in[0,t]} s^{w_{1}}\|v(s,\cdot)\|_{\ell^{s_2}(V)}\right)^{pq},\label{7.67}
\ees
where $C_{22}=B\left(1-\delta,1-wq\right)$, and
$$C_{23}=B\left(1-\delta, 1+\left(1-\delta-w_{1}p\right)q\right) B^q\left(1-\delta, 1-w_{1}p\right).$$
Let $h(t):=\sup\limits_{s\in[0,t]}\left( s^{w_1}\left\|{v(s,\cdot)}\right\|_{\ell^{s_2}(V)}\right) $. Then by \eqref{7.65} and \eqref{7.67},
$$
\begin{aligned}
    h(t) \le & C_4\left\|v_0\right\|_{\ell^{r_2}(V)}+2^{q-1} C_4^{q+1}\left\|u_0\right\|_{\ell^{r_1}(V)}^q C_{22} \\
    & +2^{q-1} C_4^{q+1} C_{23} h^{p q}(t) \\
    \leq & C_{24}\left(\left\|v_0\right\|_{\ell^{r_2}(V)}+\left\|u_0\right\|_{\ell^{r_1}(V)}^q+h^{p q}(t)\right),
\end{aligned}
$$
where
$$
C_{24}=\max \left\{C_4, 2^{q-1} C_{22} C_4^{q+1}, 2^{q-1} C_4^{q+1} C_{23}\right\}.
$$
By \eqref{7.58}, we see
$$
\left\|v_0\right\|_{\ell^{r_2}(V)}+\left\|u_0\right\|_{\ell^{r_1(V)}}^q \le\left(\hat{C}_2\right)^{\frac{1}{p}}+\hat{C}_2^q=:\hat{C}_3 .
$$
By similar discussions as in the proof of \eqref{7.60.} above,
we may show that
$$
h(t)<2 C_{24} \hat{C}_3 \text{ for } t\in(0,T_4]\text {. }
$$
This implies that
$$
\|v(t, \cdot)\|_{\ell^{s_2}(V)} \le 2 C_{24} \hat{C}_3 t^{-w_1} \text { for } t \in\left(0, T_4\right].
$$
It follows that \eqref{7.63} by Lemma \ref{l3.4}.

It follows from \eqref{6.66} and \eqref{7.63} that $(u,v)$ is global.

We now complete the proof.
\end{proof}

From Lemma \ref{l6.2}, we immediately get Theorem \ref{t2.5.}.

\section{Acknowledgement}
I would like to express my sincere gratitude to the anonymous reviewers for their meticulous review and constructive feedback on this manuscript. Their insightful comments and suggestions have significantly improved the quality, clarity, and scientific rigor of this work. I also appreciate the editor’s efforts in managing the review process and providing valuable guidance throughout.


\begin{thebibliography}{99}
    \setlength{\baselineskip}{15pt}
    
\bibitem{AW} D. G. Aronson and H. F. Weinberger, {\it Multidimensional nonlinear diffusion arising in population genetics}, Adv. Math., {\textbf{30}} (1) (1978), 33--76.  


\bibitem{BGL} D. Bakry, I. Gentil and M. Ledoux, {\it Analysis and Geometry of Markov Diffusion Operators}, Grundlehren
der Mathematischen Wissenschaften, vol. 348, Springer, Cham, 2014.

   \bibitem{BHLLMY}
    F. Bauer, P. Horn, Y. Lin, G. Lippner, D. Mangoubi and S. T. Yau, {\it Li-Yau inequality on graphs}, J. Differential Geom. {\textbf{99}} (3) (2015), 359--405.
    
    \bibitem{BHua} F. Bauer, B. Hua and J. Jost, {\it The dual Cheeger constant and spectra of infinite graphs}, Adv. Math. {\textbf {251}} (2014), 147--194.
    
%\bibitem{BK} T. B. Boykin and G. Klimeck, {\it The discretized Schrödinger equation and simple models for semiconductor quantum wells}, Eur. J. Phy. {\textbf{25}}(4) (2004), 503--514.
    
    \bibitem{BPT} C. Bandle, M. A. Pozio and A. Tesei, {\it The Fujita exponent for the Cauchy problem in the hyperbolic space}, J. Differential Equations {\textbf{251}} (8) (2011), 2143--2163.

\bibitem{BL} J. Bergh and J. L\"{o}fstr\"{o}m, {\it Interpolation spaces: an introduction},  Springer Science \& Business Media, vol. 223, 2012.

    \bibitem{BS}
    D. Bianchi, A. G. Setti and R. K. Wojciechowski, {\it The generalized porous medium equation on graphs: existence and uniqueness of solutions with $\ell^1$ data}, Calc. Var. Partial Differential Equations {\textbf{61}} (5) (2022), 171.



\bibitem{CF} I. Chavel and E. A. Feldman, {\it Modified isoperimetric constants and large time heat diffusion in Riemannian manifolds}, Duke Math. J. {\textbf{64}} (1991), 473--499.     
   %\bibitem{FC} F. R. K. Chung, {\it Spectral Graph Theory}, CBMS Reg. Conf. Ser. Math., Amer. Math. Soc., Providence, RI, 1997.
   
 \bibitem{Cowling-Giulini-Meda} M. Cowling, S. Giulini and S. Meda, {\it $L^p$-$L^q$ estimates for functions of the Laplace-Beltrami operator on noncompact symmetric spaces. I}, Duke Math. J. {\bf 65} (1993), 109--150.
   
 \bibitem{Cowling-Meda-Setti} M. Cowling, S. Meda and A. G. Setti, {\it Estimates for functions of the Laplace operator on homogeneous trees}, Trans. Amer. Math. Soc. {\textbf{352}} (2000), 4271--4293.

\bibitem{CKKLLP}D. Cushing, S. Kamtue, J. Koolen, S. Liu, F. Münch and N. Peyerimhoff, {\it Rigidity of the Bonnet-Myers inequality for graphs with respect to Ollivier Ricci curvature}, Adv. Math. {\textbf{369}}(2020) 107188.
 %\bibitem{CLC} Y.-S. Chung, Y.-S. Lee and S.-Y. Chung, {\it Extinction and positivity of the solutions of the heat equations with absorption on networks}, J. Math. Anal.  Appl. {\textbf {380}} (2011), 642--652.

\bibitem{EH}
M. Escobedo and M. A. Herrero, {\it Boundedness and blow up for a semilinear reaction-diffusion system}, J. Differential Equations {\textbf{89}} (1) (1991), 176--202.   
    
%\bibitem{EJ} J. Ch. Eilbeck and M. Johansson, {\it The discrete nonlinear Schr\"{o}dinger equation 20 years on}, In Localization and energy transfer in nonlinear systems, 44--67. World Scientific Publising, 2003.
    
%\bibitem{EM} M. Erbar and J. Maas, {\it Gradient flow structures for discrete porous medium equations}, Discrete Contin. Dyn. Syst. {\textbf{34}}(4) (2014), 1355--1374.
    
\bibitem{FHS} K. J. Falconer, J. Hu and Y. Sun, {\it Inhomogeneous parabolic equations on unbounded metric measure spaces}, Proc. R. Soc. Edinb. Sect. A: Math. {\textbf{142}} (5) (2012), 1003--1025.
    
  %\bibitem{FB1} A. Fern\'{a}ndez-Bertolin, {\it Convexity Properties of Discrete Schr\"{o}dinger evolutions and Hardy's Uncertainty Principle}, arXiv preprint arXiv:1506.03717 (2015).
    
%\bibitem{FB2} A. Fern\'{a}ndez-Bertolin,  {\it Discrete uncertainty principles and virial identities}, Appl. Comput. Harmon. Anal. {\textbf{40}}(2) (2016), 229--259.
    
%\bibitem{FB3} A. Fern\'{a}ndez-Bertolin, {\it A discrete Hardy's uncertainty principle and discrete evolutions}, J. d'Analyse Math\'{e}matique {\textbf {137}} (2019), 507--528.
    
 %\bibitem{FJ} A. Fern\'{a}ndez-Bertolin and P. Jaming, {\it Uniqueness for solutions of the Schr\"{o}dinger equation on trees}, Ann. Mat. Pura Appl. {\textbf {199}} (2020), 681--708.

 
 \bibitem{Fife} P.C. Fife, {\it  Mathematical aspects of reacting and diffusing systems}, Lecture Notes in Biomathematics 28, Springer Verlag, 1979
    
 \bibitem{Fujita} H. Fujita, {\it On the blowing up of solutions of the Cauchy problem for $u_t=\Delta u+u^{1+\alpha}$}, Fac. Sci. Univ. Tokyo Sect. A. Math. {\textbf{13}} (2) (1966), 109--124.

\bibitem{FOT} M. Fukushima, Y. Oshima,  M. Takeda, Dirichlet Forms and Symmetric Markov Processes, de
Gruyter Studies in Mathematics, vol. 19, Walter de Gruyter  Co., 2011.
%\bibitem{Gekazdan} H. Ge, {\it Kazdan-{W}arner equation on graph in the negative case}, J. Math. Anal. Appl. {\textbf {453}}(2), (2017), 1022--1027.
    
 %\bibitem{GeYa} H. Ge, {\it A $p$-th Yamabe equation on graph}, Proc. Amer. Math. Soc.  {\textbf{146}} (5) (2018), 2219--2224.
 
\bibitem{GX} H. Ge and X. Xu, {\it Discrete quasi-Einstein metrics and combinatorial curvature flows in 3-dimension}, Adv. Math. {\textbf{267}} (2014), 470--497.   
 %\bibitem{GeJiaKazdan} H. Ge and W. Jiang, {\it Kazdan-Warner equation on infinite graphs}, J. Korean Math. Soc. {\textbf {55}}(5), (2018), 1091--1101.
    
\bibitem{GLLY} C. Gong, Y. Lin, S. Liu and S. T. Yau, {\it Li-Yau inequality for unbounded Laplacian on graphs}, Adv. Math. {\textbf{357}} (2019), 106822.
    
    \bibitem{gri09} A. Grigoryan, {\it Analysis on Graphs}. Lecture Notes, University Bielefeld, 2009.
    
%\bibitem{gliny1} A. Grigor'yan, Y. Lin and Y. Yang, {\it Yamabe type equations on graphs}, J. Differential Equations {\textbf {261}} (2016), 4924--4943.
    
 %\bibitem{gliny2} A.~Grigor'yan, Y.~Lin and Y.~Yang, {\it  Kazdan-{W}arner equation on graph}, Calc. Var. Partial Differential Equations  {\textbf {55}}(92), (2016), 1--13.
    
 %\bibitem{Guiw} C. Gui and X. Wang, {\it Life span of solutions of the Cauchy problem for a semilinear heat equation}, J. Differential Equations {\textbf {115}}(1) (1995), 166--172.
    
%\bibitem{GKL} P. G\'{o}rka, A. Kurek, E. Lazarte and H. Prado, {\it Parabolic flow on metric measure spaces}, Semigroup Forum {\textbf {88}}(1) (2014), 129--144.
    
%\bibitem{HAESELER} S. Haeseler, M. Keller, D. Lenz and R. Wojciechowski, {\it Laplacians on infinite graphs: Dirichlet and Neumann boundary conditions}, J. Spectr. Theory {\textbf 2} (4) (2012), 397--432.
    
%\bibitem{Hayakawa}K. Hayakawa, {\it On nonexistence of global solutions of some semilinear parabolic differential equations}, Proc. Japan Acad. {\textbf {49}} (7) (1973), 503--505.
    
\bibitem{HL} B. Hua and Y. Lin, {\it Stochastic completeness for graphs with curvature dimension conditions}, Adv. Math. {\textbf{306}} (2017), 279-302.

\bibitem{HJL} B. Hua, J. Jost and S. Liu, {\it Geometric analysis aspects of infinite semiplanar graphs with nonnegative curvature}, J. Reine Angew. Math. 700 (2015), 1-36.

\bibitem{HuangKW} X. Huang, M. Keller, J. Masamune and R. Wojciechowski, {\it  A note
on self-adjoint extensions of the Laplacian on weighted graphs}, J. Funct. Anal. {\textbf{265}} (2013),
no. 8, 1556–1578,

\bibitem{HuangKM} X. Huang, M. Keller and M. Schmidt, {\it On the uniqueness class, stochastic
completeness and volume growth for graphs}, Trans. Amer. Math. Soc. {\textbf{373} (12)} (2023), 8861-8884.

\bibitem{HLLY} P. Horn, Y. Lin, S. Liu and S. T. Yau, {\it Volume doubling, Poincar\'{e} inequality and Gaussian heat kernel estimate for non-negatively curved graphs}, J. Reine Angew. Math. {\textbf{757}} (2019), 89--130.

 \bibitem{HuWang}  Y. Hu and M. Wang, {\it Life span of solutions to a semilinear parabolic equation on locally finite graphs}, arXiv:2405.18173 (2024).   
%\bibitem{JLMP} P. Jaming, Y. Lyubarskii, I. Malinnikova, E. and K. M. Perfekt, {\it Uniqueness for discrete Schr\"{o}dinger evolutions}, Rev. Mat. Iberoam. {\textbf {34}}(3) (2018), 949--966.
    
 %\bibitem{Kaplan}S. Kaplan, {\it On the growth of solutions of quasilinear parabolic equations}, Comm. Pure Appl. Math. {\textbf {16}} (1963), 305--333.
 \bibitem{KM} M. Kanai, {\it Rough isometries, and combinatorial approximations of geometries of non-compact Riemannian manifolds}, J. Math. Soc. Japan {\textbf{37}} (3) (1985), 391--413.
    
    \bibitem{KL} M. Keller and D. Lenz, {\it Dirichlet forms and stochastic completeness of graphs and subgraphs}, J. Reine Angew. Math. {\textbf {666}} (2012), 189--223.
    
%\bibitem{KS} M. Keller and M. Schwarz, {\it The Kazdan-Warner equation on canonically compactifiable graphs}, Calc. Var. Partial Differential Equations {\textbf {57}} (2018), 1--18.
    
%\bibitem{Kobayashi} K. Kobayashi, T. Sirao and H. Tanaka, {\it On the growing up problem for semilinear heat equations}, J. Math. Soc. Japan {\textbf {29}} (3) (1977), 407--424.
    
    
    \bibitem{LCE} L. C. Evans, {\it Partial differential equations}, Vol. 19. American Mathematical Society, 2022.
    
    \bibitem{Len} D. Lenz, M. Schmidt and I. Zimmermann, {\it Blow-up of nonnegative solutions of an abstract semilinear heat equation with convex source},  Calc. Var. Partial Differential Equations {\textbf{62}} (4) (2023), 1--19.
    
 % \bibitem{LWjm} Y. Lin and Y. Wu, On-diagonal lower estimate of heat kernels on graphs, J. Math. Anal. Appl. {\textbf{456}}(2) (2017), 1040--1048.
\bibitem{Li} Y. Y. Li, {\it Harnack type inequality: the method of moving planes}, Commun. Math. Phys. {\textbf{200}} (2)
  (1999), 421--444.
  
 \bibitem{LWcv} Y. Lin and Y. Wu, {\it The existence and nonexistence of global solutions for a semilinear heat equation on graphs}, Calc. Var. Partial Differential Equations {\textbf {56}} (2017), 1--22.
 
 \bibitem{LZ}A. Lin and X. Zhang, {\it Combinatorial p-th Calabi flows on surfaces}, Adv. Math. {\textbf {346}} (2019), 1067--1090.
    
 %\bibitem{LWac} Y. Lin and Y. Wu, {\it Blow-up problems for nonlinear parabolic equations on locally finite graphs}, Acta Math. Scientia {\textbf {38}}(3) (2018), 843--856.
    
% \bibitem{LZ} X. Liang and T. Zhou, {\it Spreading speeds of KPP-type lattice systems in heterogeneous media}, Commun. Contemp. Math. \textbf{22}(01) (2020), 1850083.
    
 %\bibitem{Mu} D. Mugnolo, {\it Parabolic theory of the discrete p-Laplace operator}, Nonlinear Anal. {\textbf {87}} (2013), 33--60.
    
    
 %\bibitem{OY2011} T. Ozawa and Y. Yamauchi, {\it Life span of positive solutions for a semilinear heat equation with general non-decaying initial data}, J. Math. Anal. Appl. \textbf{379}(2) (2011), 518--523.
    
  %\bibitem{PW} M. H. Protter and H. F. Weinberger, {\it Maximum Principles in Differential Equations}, Prentice-Hall, Englewood Cliffs, NJ, 1967.
    
 %\bibitem{PS} A. Pinamonti and G. Stefani, Existence and uniqueness theorems for some semi-linear equations on locally finite graphs, Proc. Amer. Math. Soc.  {\textbf{150}}(11) (2022), 4757--4770.
\bibitem{MM} M. M. Pang, {\it Heat kernels of graphs}, J. London Math. Soc. {\textbf{47}} (1993), 50--64.
  	  
    \bibitem{RY}  M. Ruzhansky and N. Yessirkegenov, {\it Existence and non-existence of global solutions for semilinear heat equations and inequalities on sub-Riemannian manifolds, and Fujita exponent on unimodular Lie groups}, J. Differential Equations {\textbf{308}} (2022), 455--473.
 
\bibitem{SK}N. Shigesada and K. Kawasaki, {\it Biological invasions: theory and practice}, Oxford Series in
 Ecology and Evolution, Oxford, Oxford University Press, 1997
  
\bibitem{SS} B. Shorrocks and I.R. Swingland, {\it Living in a Patch Environment}, Oxford University Press, NewYork, 1990 
  
\bibitem{Sun} L. Sun and L. Wang, {\it Brouwer degree for Kazdan-Warner equations on a connected finite graph}, Adv. Math. {\textbf{404}} (2022), 108422.
  
\bibitem{Taylor} M. E. Taylor, {\it $L^p$-estimates on functions of the Laplace operator}, Duke Math. J. {\textbf{58}} (3) (1989), 773--793.  
  %\bibitem{Tian} C. Tian, Z. Liu and S. Ruan, {\it Asymptotic and transient dynamics of SEIR epidemic models on weighted networks}, Eur. J. Appl. Math. {\textbf {34}}(2) (2023), 238--261.
    
%\bibitem{AW} A. Weber, {\it Analysis of the physical Laplacian and the heat flow on a locally finite graph}, J. Math. Anal. Appl. {\textbf {370}}(1) (2010), 146--158.
\bibitem{V} N. T. Varopoulos, {\it Brownian motion and random walks on manifolds}, In Annales de l'institut Fourier {\textbf{34}} (2) (1984), 243--269.


\bibitem{Wang} X. Wang, {\it On the Cauchy problem for reaction-diffusion equations}, Trans. Amer. Math. Soc. {\textbf {337}} (2) (1993), 549--590.
    
\bibitem{W1} F. Weissler, {\it Local existence and nonexistence for semilinear parabolic equations in $L^p$}, Indiana Univ. Math. J. {\textbf {29}} (1) (1980), 79--102.
    
\bibitem{W2}F. Weissler, {\it Existence and non-existence of global solutions for a semilinear heat equation}, Israel J. Math. {\textbf {38}} (1981), 29--40.
    
\bibitem{Wu} Y. Wu, {\it Blow-up conditions for a semilinear parabolic system on locally finite graphs}, Acta Math. Sci. {\textbf {44}} (2) (2024), 609--631.
    
\bibitem{RW}R. Wojciechowski, {\it Heat kernel and essential spectrum of infinite graphs}, Indiana Univ. Math. J. {\textbf{58}} (3) (2009), 1419--1442.
    
 % \bibitem{WY} Z. Wang and J. Yin, {\it Asymptotic behaviour of the lifespan of solutions for a semilinear heat equation in hyperbolic space}, Proc. Royal Soc. Edinburgh A {\textbf {146}}(5) (2016), 1091--1114.
    
 %\bibitem{W} Y. Wu, {\it On nonexistence of global solutions for a semilinear heat equation on graphs}, Nonlinear Anal. {\textbf{171}} (2018), 73--84.
    
 %\bibitem{Wc} Y. Wu, {\it Blow-up for a semilinear heat equation with Fujita's critical exponent on locally finite graphs}, RACSAM {\textbf{115}}(3) (2021), 133.
    
 %\bibitem{XinMu} Q. Xin, L. Xu and C. Mu, {\it Blow-up for the $\omega$-heat equation with Dirichelet boundary conditions and a reaction term on graphs}, Appl. Anal. {\textbf {93}}(8) (2014), 1691--1701.
    
%\bibitem{Ya2011} Y. Yamauchi, {\it Life span of solutions for a semilinear heat equation with initial data having positive limit inferior at infinity}, Nonlinear Anal. {\textbf{74}}(15) (2011), 5008--5014.
\bibitem{XMY} Q.H. Xu, T. Ma and Z. Yin,  {\it Lectures on Functional Analysis. 1st ed.}, Beijing: Higher Education Press, 2017.

\bibitem{ZQ} Q. S. Zhang. {\it Blow-up results for nonlinear parabolic equations on manifolds}, Duke Math. J.  {\textbf {97}} (1999), 515–539.    
    
%\bibitem{zl1Pro} X. Zhang and A. Lin, {\it Positive solutions of $p$-th Yamabe type equations on infinite graphs}, Proc. Amer. Math. Soc. {\textbf{147}}(4) (2019), 1421--1427.
    
%\bibitem{zl2Fr} X. Zhang and A. Lin, {\it Positive solutions of $p$-th Yamabe type equations on graphs}, Front. Math. China \textbf{13} (2018), 1501--1514.
\end{thebibliography}
\end{document}